\documentclass[11pt,reqno]{amsart}

\usepackage[utf8]{inputenc}
\usepackage{amsmath, amssymb, amsthm, mathrsfs}
\usepackage{bm,bbm}
\usepackage{csquotes}
\usepackage{stmaryrd}
\usepackage{ulem}
\usepackage{lipsum}

\usepackage[colorlinks]{hyperref}
\usepackage[capitalize]{cleveref}
\usepackage{graphicx}
\usepackage{enumerate}
\usepackage{enumitem}
\usepackage{mathtools}

\usepackage{xcolor}
\usepackage{float}
\usepackage{svg}

\usepackage[margin=1in]{geometry}

\newcommand{\RR}{\mathbb{R}}

\newcommand{\SCN}{\mathcal{N}}
\newcommand{\PP}{\mathbb{P}}

\newcommand{\EE}{\mathbb{E}}
\newcommand{\UU}{\mathcal{U}}
\newcommand{\LL}{\text{Law}\,}
\newcommand{\AAA}{\mathbb{A}}
\newcommand{\defeq}{\vcentcolon =}
\newcommand{\from}{\colon}
\newcommand{\frm}{\mathfrak{m}}

\newtheorem{thm}{Theorem}[section]
\newtheorem{lem}[thm]{Lemma}
\newtheorem{prop}[thm]{Proposition}

\theoremstyle{definition}

\newtheorem{ex}[thm]{Example}
\newtheorem{rk}[thm]{Remark}
\newtheorem{asmp}[thm]{Assumption}

\title[Uniform-in-Time Convergence Rates for Mean-Field Interacting Jump Processes]{
Uniform-in-Time Convergence Rates to a Nonlinear Markov Chain for Mean-Field Interacting Jump Processes}\thanks{* This is the final version of the paper. To appear in {\it SIAM Journal on Control and Optimization}}

\author[A. Cohen]{Asaf Cohen }
\address{Department of Mathematics\\
University of Michigan\\
Ann Arbor, MI 48109\\
United States
}
\email{shloshim@gmail.com}

\author[E. Huffman]{Ethan Huffman}
\address{Department of Mathematics\\
University of Michigan\\
Ann Arbor, MI 48109\\
United States
}
\email{ehuf@umich.edu}

\begin{document}

\maketitle

\begin{abstract}
	We consider a system of $N$ particles interacting through their empirical distribution on a finite state space in continuous time. In the formal limit as $N\to\infty$, the system takes the form of a nonlinear (McKean--Vlasov) Markov chain. This paper rigorously establishes this limit. Specifically, under the assumption that the mean field system has a unique, exponentially stable stationary distribution, we show that the weak error between the empirical measures of the $N$-particle system and the law of the mean field system is of order $1/N$ uniformly in time.
	Our analysis makes use of a master equation for test functions evaluated along the measure flow of the mean field system, and we demonstrate that the solutions of this master equation are sufficiently regular.
	We then show that exponential stability of the mean field system is implied by exponential stability for solutions of the linearized Kolmogorov equation with a source term.
	Finally, we show that our results can be applied to the study of mean field games and give a new condition for the existence of a unique stationary distribution for a nonlinear Markov chain.

    \vspace{5pt}

    %
    %
    %
    %

\end{abstract}

\section{Introduction}

This paper considers the long-time behavior of a system of $N$ interacting particles on a finite state space $[d] \defeq \{1, \dots d\}$.
The state of each particle $(X^{i,N}_t)_{t\ge0}$, $i\in[N]:=\{1,\ldots,N\}$, evolves by jumping to a new state according to a rate matrix $\alpha(\mu^N_t)=(\alpha_{xy}(\mu^N_t))_{x,y\in[d]}$, where $\mu^N_t$ is the empirical distribution of the system at time $t$.
More precisely, we have
\begin{equation*}
	\PP(X^{i,N}_{t+\varepsilon} = y \mid X^{i,N}_t = x, \mu^N_t = \mu) = \alpha_{xy}(\mu) \varepsilon + o(\varepsilon),\qquad x\ne y.
\end{equation*}
The empirical distribution of the system is a Markov chain which evolves according to
\begin{equation*}
	\PP(\mu^N_{t+\varepsilon} = \mu + \tfrac{1}{N} (\delta_x - \delta_y) \mid \mu^N_t = \mu) = N \mu_x \alpha_{xy}(\mu) \varepsilon + o(\varepsilon),\qquad x\ne y,
\end{equation*}
where $\delta_z$ is the $z$th standard basis vector in $\RR^d$  and $\mu_z$ denotes $\mu(\{ z \})$. In the limit as $N \to \infty$, we expect to have a single representative particle $Y_t$, whose dynamics are given by
\begin{equation*}
	\PP(Y_{t+\varepsilon} = y \mid Y_t = x) = \alpha_{xy}(\LL(Y_t)) \varepsilon + o(\varepsilon),\qquad x\ne y.
\end{equation*}
We refer to this limiting system as the mean field system and notice that $Y_t$ is a \textit{nonlinear Markov chain}, that is, a Markov chain whose transition rates depend on the law of the process.
The law of the mean field system, $(m(t; \mu))_{t \geq 0}$, obeys the (nonlinear forward) Kolmogorov equation
\begin{equation}\label{eq:kolmogorovIntro}
	\frac{d}{dt} m(t; \mu) = m(t; \mu) \alpha(m(t; \mu)), \quad m(0;\mu) = \mu,
\end{equation}
which is posed on the simplex, $S_d$, of probability measures on $[d]$.

The main result of this paper shows that the empirical measures of the $N$-particle system converge to the law of the mean-field system uniformly in time, in a weak sense. This convergence holds when the mean-field system has a unique, exponentially stable stationary distribution, and the derivatives of the mean-field system with respect to initial conditions are uniformly Lipschitz in the measure argument. 
Namely, for any sufficiently regular test function $\Phi \from S_d \to \RR$, we show that
\begin{equation}\label{eq:mainResIntro}
	 \sup_{t \geq 0}\left\lvert \EE \left[ \Phi(\mu^N_t)\right] - \Phi(m(t; \mu))   \right\rvert \leq \frac{C}{N}.
\end{equation}
The prototypical example of $\Phi$ is the squared Euclidean distance, $\lvert \cdot  - \nu_\infty \rvert_2^2$, for $\nu_\infty$ an attracting stationary distribution of \eqref{eq:kolmogorovIntro}  in the case that such a distribution uniquely exists.
Denoting by $\mu^N_\infty$ the stationary distribution of the system $(X^{i,N}_t)_{i \in [N]}$ and letting $t \to \infty$, this $\Phi$ gives us that
\begin{equation*}
    \EE \left[ \lvert \mu^N_\infty - \nu_\infty \rvert_2^2 \right] \leq \frac{C}{N},
\end{equation*}
which recovers Equation (2) in \cite{Ying2018}. Our result is stronger, as it establishes a uniform bound on the distance between the $N$-particle system's distribution and the limiting distribution for all time instances $t$, rather than only in the stationary regime.

Additionally, as it can be difficult to verify the exponential stability of the mean field system, we show that the exponential stability of all solutions to a linear ordinary differential equation, which we call the linearized Kolmogorov equation, implies the exponential stability of the mean field system.
This implication allows us to obtain a new condition for the existence of a unique, exponentially stable stationary distribution for a nonlinear Markov chain which is simple to check and relatively unrestrictive.
In fact, we obtain further that, under these conditions, derivatives of the nonlinear Markov chain with respect to the measure are Lipschitz in their measure argument uniformly in time. 
Finally, we discuss the application of our results to the convergence problem in mean field games, showing that our results apply to an $N$-player game with players using controls derived from the ergodic master equation as defined in \cite{coh-zel2022, coh-zel24}.

\subsection{Background}

    Delarue and Tse recently demonstrated in \cite{del-tse2021} that the empirical measures of a weakly interacting $N$-particle system on the torus converge weakly to the law of a McKean--Vlasov process, with a convergence rate of $1/N$.
    Notably, this convergence is uniform in time.
    In order to show this result, Delarue and Tse use a PDE on the space of probability measures satisfied by the semigroup of the McKean--Vlasov process.
    They call this PDE the \textit{master equation}, which we note is a title applied to many different PDEs in the literature, depending on the context.
    The analysis in \cite{del-tse2021} connects linearized Fokker--Planck type equations to derivatives of the master equation, and applies an ergodicity assumption on the solutions to these linearized equations.
    We follow a similar approach, but in the finite state case it is possible to work with an ergodicity assumption directly on the nonlinear Kolmogorov equation, which we show is weaker than the corresponding assumption on linearized Kolmogorov equations.
    Additionally, we are able to work with less regularity than in the diffusion case, requiring only one Lipschitz continuous derivative for our test functions.

Propagation of chaos is a challenging result to achieve uniformly in time.
Indeed, Malrieu \cite{Malrieu2003} showed by an example of a diffusion on Euclidean space that even when the limiting process has a unique, globally attracting stationary distribution, uniform convergence in time may not hold.
Although there are several results concerning uniform-in-time propagation of chaos for diffusion processes on Euclidean space where the drift term originates from a potential \cite{Malrieu2001, Guillin2023, Rosenzweig2023, Chen2024, Durmus2020, Salem2020}, results beyond this setting remain sparse.
Recently, Lacker and Le Flem \cite{LackerLeFlem2023} extended previous work of Lacker \cite{Lacker2023} and, using the BBGKY hierarchy, established a sharp rate for uniform-in-time convergence in relative entropy for the joint law of the first $k$ of $N$ particles  when the interaction of the particles is pairwise.
More precisely, they showed convergence to the $k$-fold product of the limiting process’s law at rate $(k/N)^2$, which implies convergence in the Wasserstein metric at rate $(k/N)$.

In the finite state case, Ying \cite{Ying2018} used Stein's method to  analyze systems with mean-field particle interactions, showing that the squared distance between the {\it stationary distribution} of the $N$-particle system and that of the mean-field system converges to zero at rate $1/N$.
Budhiraja, Dupuis, Fischer, and Ramanan \cite{Budhiraja2015, Budhiraja2015_2} used the relative entropy of the mean-field system with respect to the $N$-particle system’s stationary distribution as a Lyapunov function, establishing stability for mean-field systems in cases of slow adaptation or locally Gibbs-type interactions.
Going beyond the case of a single, asymptotically stable stationary distribution to models with simple $\omega$-limit sets, Borkar and Sundaresan \cite{bor-son2012} derived a large deviation principle, following the approach of Freidlin and Wentzell \cite{fre-wen2012}, on the simplex for the stationary distributions of the empirical measures of $N$-particle systems.

Nonlinear Markov chains, initially introduced by McKean \cite{McKean1966} in the context of plasma physics, were later studied in general by Kolokoltsov \cite{kolokoltsov2010nonlinear}.
The long-term behavior of nonlinear Markov chains is significantly more complex than that of linear Markov chains, leading to limited results on the uniqueness of stationary distributions and ergodicity.
In discrete time, Butkovsky \cite{Butkovsky2014} showed that an analog of Dobrushin’s condition, under small interactions, ensures a unique, globally attracting stationary distribution.
Saburov \cite{Saburov2016} obtained a similar result using hypermatrices.
Also in discrete time, Light \cite{Light2023} has shown that for nonlinear Markov chains with an aggregator, there are monotonicity conditions that ensure the existence of a unique stationary distribution.
In continuous time, Neumann \cite{Neumann2022} has provided a condition for the existence of a unique stationary distribution that uses the invariant distributions of each $\alpha(\mu)$ separately.
Additionally, Neumann has shown that for $d = 2,3$ there are conditions for a unique stationary distribution to be asymptotically ergodic.

Mean field games (MFGs) were introduced as a limiting model of symmetric $N$-player stochastic differential games by Lasry and Lions \cite{Lasry2007} and independently by Huang, Malham\'e, and Caines \cite{Huang2006}. MFG models approximate Nash equilibria in large-player games, where the concept of a mean field equilibrium serves as the MFG analog of a Nash equilibrium. This equilibrium is defined via a fixed point, which can be characterized either probabilistically or by two differential equations: a forward equation describing the evolution of the measure under equilibrium, and a backward equation for the value function, which is a Hamilton--Jacobi--Bellman equation. The forward-backward system can be encapsulated in a single differential equation, 
sometimes known as the master equation, but called the {\it MFG master equation} in the following to reduce confusion with the master equation of interacting particle systems.

Finite-state MFGs with finite time horizons have been extensively studied, including the existence of the master equation and its application in analyzing the game \cite{bay-coh2017, cec-pel2019, Cohen-Zell-Lauriere}, and adding common noise \cite{Bayraktar2021, Bayraktar2022, Cecchin2021}. Recent work has also focused on the long-term behavior of MFGs. Cardaliaguet and Porretta \cite{car-por2019} studied ergodic MFGs of diffusions on the torus, under the monotonicity condition (interpreted as players seeking to avoid congestion), analyzing the MFG master equation and proving the existence and uniqueness of a weak solution. In the finite-state setting, Gomes, Mohr, and Souza \cite{Gomes2013} introduced the ergodic MFG system and, under contractivity assumptions, demonstrated that it admits a unique solution. Cohen and Zell \cite{coh-zel2022} extended this result by removing the contractivity requirement, assuming monotonicity, and introducing an ergodic MFG master equation for finite-state MFGs, proving both the existence and uniqueness of its solution. In a subsequent work \cite{coh-zel24}, they showed that controls derived from the MFG system or the MFG master equation yield a $(C/\sqrt{N})$-Nash equilibrium for the corresponding $N$-player game.
Recently, H\"ofer and Soner \cite{Hofer2024} considered a two-state game with periodic controls and an anti-monotone cost, demonstrating that the ergodic system can have infinitely many mean field equilibria.

\subsection{Contributions}
In \Cref{sec:mainResult} we introduce our model and main results.
We operate primarily under the ergodicity assumption \textbf{(Erg)}, which states that the mean field system is exponentially stable and has Lipschitz derivatives of its law. 
As this condition is somewhat difficult to verify, we also include the simpler to check assumption \textbf{(Lin-Erg)}, which requires that solutions to the Cauchy problems for the linearized Kolmogorov equations have exponential decay, and later we show that this implies \textbf{(Erg)}.
To illustrate the importance of these assumptions, we provide examples of systems that display counterintuitive behavior in \cref{ex:nonErg} and \cref{ex:slowConv}.

We start by giving rigorous definitions of the dynamics for the $N$-particle system 
\begin{equation*}
  \mathbf{X}^N_t=(X^{1,N}_t,\ldots,X^{N,N}_t),  
\end{equation*}
and the mean field system, $Y_t$, and present our main result, \cref{thm:mainRes}, which for a sufficiently regular test function $\Phi$ on the simplex, gives that \eqref{eq:mainResIntro} holds uniformly in time when the initial distribution of each particle $X^{i,N}_0$ is independent and equal to the initial distribution of the mean field system, which we denote by $\mu$.
In order to prove this result, we introduce in \cref{prop:master} the function $\UU(t, \mu) \defeq \Phi(m(t; \mu))$ and show that it satisfies the \textit{master equation}
\begin{equation*}
	\frac{\partial \UU}{\partial t}(t, \mu) = \sum_{z \in [d]} \frac{\delta \UU}{\delta \frm}(t, \mu, z) (\mu \alpha(\mu))_z, \quad \UU(0, \mu) = \Phi(\mu),
\end{equation*}
 where $\frac{\delta \UU}{\delta \frm}$ is the linear functional derivative we introduce below in \eqref{eq:linFuncDer}.
The proof of \cref{thm:mainRes} consists of two main components. 
Firstly, we show in \cref{prop:firstTermBound} that the error introduced by evaluating $\Phi$ at $m(t; \mu^N_0)$ instead of $\mu^N_t$ is of order $1/N$ uniformly in time.
Intuitively, this is handling the difference between propagating the test function along the measure flow of the mean field system and that of the $N$-particle system.
Secondly, in \cref{prop:secondTermBound}, we show that the error coming from the difference between $m(t; \mu^N_0)$ and $m(t; \mu)$ is also of order $1/N$ uniformly in time.
This step shows that the approximation of the initial distribution of the mean field system by the initial empirical distribution of the $N$-particle system, which is $N$ independent samples from $\mu$, also results in an error of order $1/N$.

\Cref{sec:mainResProof} contains the proof of \cref{thm:mainRes} proceeding by the subresults outlined above.
First, we show the necessary regularity of the solution to the master equation in \cref{prop:master}.
By applying the master equation to a form of It\^o's formula for the $N$-particle system we are able to obtain a tractable expression for the time propagation error in \cref{prop:firstTermExp}.
This expression then allows us to complete the proof of \cref{prop:firstTermBound}.
Finally, to prove \cref{prop:secondTermBound}, we directly express the sampling error in terms of derivatives of the solution of the master equation, and apply the regularity results from \cref{prop:master}.


In \Cref{sec:linErgToErg} we show that the assumption \textbf{(Erg)} used for our main result can be replaced by the assumption \textbf{(Lin-Erg)}.
Specifically, in \cref{prop:mConv}, we show that the exponential stability of all solutions to the linearized Kolmogorov equation, potentially with a source term, implies the exponential stability of the mean field system.
This result leverages the connection between particular solutions of the linearized Kolmogorov equation and derivatives of the nonlinear Kolmogorov equation as well as the properties of the linear functional derivative.

In \Cref{sec:applications} we turn our results towards applications.
\cref{prop:mConv} allows us in \Cref{subsec:nonlinearMarkov} to demonstrate two new and particularly simple conditions for the existence of unique and exponentially stable stationary distributions for nonlinear Markov chains, as well as further regularity of the solution trajectories. 
Then, in \Cref{subsec:MFG}, we show that our results are also applicable to the convergence of finite state large population games with mean field interactions and long-term average (ergodic) costs.
Indeed, the exact conditions which are common in the mean field game literature to ensure the existence of unique mean field equilibria allow us to apply our main result.
Our result extends that of Cohen and Zell \cite{coh-zel24} by showing that propagation of chaos applies to the $N$-player game with controls derived from the master equation uniformly in time, rather than just with the stationary distribution of the mean field system.

We finish the present section with an overview of our notation and some basic definitions and properties that will be used throughout the paper.

\subsection{Notation}

We fix positive integers $d$ and $N$ and use the notation $[d] \defeq \{1, \dots, d\}$ and $[N] \defeq \{1, \dots, N\}$.
Indices $x, y, z$, etc. will take values in $[d]$, while indices $i,j,k$, etc. will generally take values in $[N]$.
 For a vector $v \in \RR^d$, we will typically denote the $x$th component of $v$ by $v_x$, while for a matrix $A \in \RR^{d \times d}$, the notation $A_x$ will be used to denote the $x$th row of $A$.

We denote by $\RR^d_0$ the set of all $v \in \RR^d$ such that $\sum_z v_z = 0$.
The $(d-1)$-simplex considered as a closed subset of $\RR^d$ is denoted by $S_d$, and we identify the space of probability measures on $[d]$ with $S_d$.
For any vector $v \in \RR^d$, we denote by $|v|$ the $L^1$ norm $\sum_{x \in [d]} |v_x|$.

The vector $\delta_z$ is the $z$th standard basis vector in $\RR^d$, which we will often think of as the Dirac measure at $z$, and we let $\delta_{yz} \defeq \delta_z - \delta_y$.

For $\overline{x} = (x_1, \dots, x_N) \in [d]^N$, we denote the empirical measure
\begin{equation*}
	\mu^N_{\overline{x}} \defeq \frac{1}{N} \sum_{i=1}^N \delta_{x_i}.
\end{equation*}
For $z \in [d]$ and $q \in \RR^d$, we denote
\begin{equation*}
	\Delta_z q \defeq (q_y - q_z)_{y \in [d]}.
\end{equation*}
Additionally, for a function $F \from [d]^N \to \RR$, we denote by $\Delta^i F(\mathbf{x})$, the vector of differences of $F$ in the $i$th coordinate, i.e.
\begin{equation*}
	\Delta^i F(\mathbf{x}) \defeq \big( F(x_1, \dots, x_{i-1}, 1, x_{i+1}, \dots, x_N) - F(\mathbf{x}), \dots, F(x_1, \dots, x_{i-1}, d, x_{i+1}, \dots, x_N) - F(\mathbf{x}) \big).
\end{equation*}
Throughout, we will use $C$ to denote a generic constant, which may change from line to line.
Wherever $C$ depends on a parameter we will mention this explicitly.

The space
\[
\tilde{S}_d\defeq\Big\{m \in [0,\infty)^{d-1}: \sum_{z \in [d-1]}m_z \leq 1\Big\}.
\]
is naturally isomorphic to $S^{d}$.
The isomorphism is 
\begin{equation*}
	\pi \from S_d \to \Tilde{S}^d, \quad \pi(m_1, \dots, m_d) := (m_1, \dots, m_{d-1}),
\end{equation*}
which has inverse $\pi^{-1}(m_1, \dots, m_{d-1}) = (m_1, \dots, m_{d-1}, 1 - \sum_{x = 1}^{d-1} m_x)$.
For a function $F \from S_d \to U$, where $U$ is an arbitrary metric space, we define the function $\tilde{F} \from \tilde{S}_d \to U$ by
\begin{equation*}
	\tilde{F}(\tilde{m}) \defeq F(\pi^{-1}(\tilde{m}))
\end{equation*}

In our analysis, it will frequently be necessary to take derivatives of functions defined on $S_d$, for which there are several different, yet related, notions in the literature.
We define two types of derivatives for a function $F \colon S_d \to \RR$.
First, for $y, z \in [d]$ and $\mu \in S_d$, we define:
\begin{equation*}
	D^\frm_{yz}F(\mu) \defeq \lim_{\varepsilon \to 0^+} \frac{F(\mu + \varepsilon \delta_{yz}) - F(\mu)}{\varepsilon},
\end{equation*}
whenever the expression within the limit is valid and the limit exists. 
This derivative is simply the restriction of the directional derivative from $\RR^d$ in the direction of $\delta_{yz}$.
We also define $D^\frm_y \defeq (D^\frm_{yz})_{z \in [d]}$, which we think of as a row vector.
Second, for $z \in [d]$ and $\mu \in S_d$, we define:
\begin{equation}\label{eq:linFuncDer}
	\frac{\delta F}{\delta \frm}(\mu, z) \defeq \lim_{\varepsilon \to 0^+} \frac{F((1-\varepsilon) \mu + \varepsilon \delta_z) - F(\mu)}{\varepsilon},
\end{equation}
also whenever the limit exists.
This second notion of derivative aligns with the linear {\it functional derivative} used in much of the literature dealing with diffusion processes.
A useful feature of the linear functional derivative is that it respects the geometry of the simplex, in the sense that it naturally restricts to the tangent space of the simplex at $\mu$.
By passing to $\Tilde{S}_d$ it is straightforward to see that
\begin{equation}\label{eq:derRel0}
	\frac{\delta F}{\delta \frm}(\mu, z) = D^\frm_1 F(\mu) \cdot (\delta_z - \mu).
\end{equation}
Using \eqref{eq:derRel0}, we can easily see that the linear functional derivative satisfies the normalization condition
\begin{equation}\label{eq:funcDerNorm}
	\sum_{z \in [d]} \frac{\delta F}{\delta \frm}(\mu, z) \mu_z = 0.
\end{equation}

To pass between these concepts of derivatives we use that, by \eqref{eq:derRel0}, we have
\begin{equation}\label{eq:derRel1}
	D^\frm_{yz}F(\mu) = \frac{\delta F}{\delta \frm}(\mu, z) - \frac{\delta F}{\delta \frm}(\mu, y)
\end{equation}
and by \eqref{eq:derRel0} and \eqref{eq:derRel1} we have
\begin{equation}\label{eq:derRel2}
	\sum_{y \in [d]} D^\frm_{yz}F(\mu) \mu_y = \frac{\delta F}{\delta \frm}(\mu, z).
\end{equation}

Taking $\mu_0, \mu_1 \in S_d$, we have the fundamental theorem of calculus:
\begin{equation}\label{eq:differenceIntegral}
	F(\mu_1) - F(\mu_0) = \int_0^1 \sum_{z \in [d]} \frac{\delta F}{\delta \frm}\big((1-\zeta) \mu_0 + \zeta \mu_1, z\big) (\mu_1 - \mu_0)_z \; d\zeta.
\end{equation}

In the case that we have a function $F \from S_d \to \RR^n$, we define the derivative $\frac{\delta F}{\delta \frm}(\mu, z) \in \RR^n$, to be the vector of derivatives of the components of $F$, i.e.
\begin{equation*}
	\left( \frac{\delta F}{\delta \frm}(\mu, z) \right)_i \defeq \frac{\delta F_i}{\delta \frm}(\mu, z).
\end{equation*}
Functions $F \from S_d \to \RR^{n \times n}$ are handled likewise.
In the case that $F$ is a function of the time $t$ as well, the linear functional derivative is defined as above with $t$ fixed, and denoted by $\frac{\delta F}{\delta \frm}(t, \mu, z)$.

\section{The Model and the Main Result}\label{sec:mainResult}

In \Cref{subsec:model}, we introduce the model under consideration, describing the dynamics of the $N$ particles and the dynamics of a nonlinear Markov chain $X_t$ as well as the evolution of its distribution $m(t; \mu)$ according to the nonlinear Kolmogorov equation. 
We also define the linearized Kolmogorov operator $L_m$ and describe the auxiliary Cauchy problems which we use throughout our analysis. 
The assumptions under which we will work are presented in \Cref{subsec:assumptions}.
Finally, in \Cref{subsec:mainResult}, we present our main result, Theorem \ref{thm:mainRes}, which establishes the bound \eqref{eq:mainResIntro}, and give a brief overview of the proof, including introducing the master equation for $\UU(t, \mu)$.

\subsection{Finite State System Setup}\label{subsec:model}

We consider processes with state space $[d]$ for $d > 1$ an integer.
The transition rates are of the form $\alpha \from S_d \to \RR^{d \times d}$, where, for all $\mu \in S_d$, $\alpha_{xy}(\mu) \in [0,\infty)$ for $x \neq y \in [d]$ and $\sum_y \alpha_{xy}(\mu) = 0$.
The notation $\alpha_x(\mu)$ will be used to denote the row of $\alpha(\mu)$ corresponding to $x$.
Occasionally, where notationally convenient, we will write $\alpha_{xy}(\mu)$ as $\alpha_y(x, \mu)$.
We will always assume that $\alpha$ is a continuous function, and denote
\begin{equation*}
    M \defeq \max_{\mu \in S_d} \,  \max_{x \neq y \in [d]} \, \lvert \alpha_{xy}(\mu) \rvert,
\end{equation*}
which is guaranteed to be finite by the compactness of $S_d$.

We rigorously adopt the representation of dynamics introduced by Cecchin and Fischer \cite{Cecchin2017}, in which the dynamics of particle $i$ are described by the following stochastic differential equation:
\begin{equation}
	\label{eq:markov}
	X^{i,N}_t = \xi^i + \int_0^t \int_{\AAA^d} \sum_{y \in [d]} \left( y - X^{i,N}_{s^-} \right) \mathbf{1}_{ \{ u_y \in (0, \alpha_y(X^{i,N}_{s^-}, \mu^N_{s^-})) \} } \SCN^{i, N}(ds, du),
\end{equation}
where the $\xi^i$ are $[d]$-valued random variables, and we abuse notation by abbreviating $\mu^N_s = \mu^N_{\mathbf{X}^N_s}$ for the empirical distribution of $\mathbf{X}^N_s \defeq (X^{1,N}_s, \dots, X^{N,N}_s)$.
The independent Poisson random measures $(\SCN^{i,N})_{i \in [N]}$ have intensity measure $\nu$, defined by
\begin{equation*}
	\nu(E) \defeq \sum_{y \in [d]} \operatorname{Leb}(E \cap \AAA^d_y),
\end{equation*}
where $\AAA^d_y \defeq \{ u \in [0, M]^d \;|\; u_x = 0 \text{ for all } x \neq y \}$, and $\operatorname{Leb}$ is the one-dimensional Lebesgue measure on the set $\{u \in \RR^d \;|\; u_x = 0 \text{ for all } x \neq y \}$.
The element $\alpha_{xy}(\mu)$ gives the transition rate of the process from state $x$ to state $y$ given the distribution $\mu$.

As the number of particles goes to infinity with the $\xi^i$ i.i.d, we expect \eqref{eq:markov} to converge to a McKean--Vlasov process described by
\begin{equation*}
	Y_t = \xi + \int_0^t \int_{\AAA^d} \sum_{z \in [d]} \left( z - Y_s \right) \mathbf{1}_{ \{ u_z \in (0, \alpha_z(Y_s, \LL(Y_s)) \} } \SCN(ds, du).
\end{equation*}

We denote the distribution of $Y_t$ starting from an initial distribution $\mu$ by $m(t; \mu)$, which evolves according to the {\it Kolmogorov equation}:
\begin{equation}\label{eq:kolmogorov}
	\frac{d}{dt} m(t; \mu) = m(t; \mu) \alpha(m(t; \mu)), \quad m(0;\mu) = \mu,
\end{equation}
where here and in the following, all $d$-dimensional vectors are considered as row vectors.

To follow \cite{del-tse2021}, given $\eta \in S_d$ we define the {\it linearized Kolmogorov operator} $L_\eta: \RR^d_0 \to \RR^d_0$
\begin{equation*}
	L_\eta q \defeq \eta \left( \sum_{z \in [d]} \frac{\delta \alpha}{\delta \frm}(\eta, z) q_z \right) + q \alpha(\eta)
\end{equation*}

We will consider Cauchy problems of the form
\begin{equation*}
	\frac{d}{dt} q(t) - L_{m(t; \mu)}q(t) - r(t) = 0,
\end{equation*}
with $q(0) = q_0$ for a \textit{source term} $r \from \RR^+ \to \RR^d_0$.
We denote this problem by \textbf{(Linear-[$\mu, q_0, r$])} and call it the \textit{linearized Kolmogorov equation}.
Note that this has the same form as the original Kolmogorov equation but with the nonlinear term replaced by the linear operator $L_{m(t, \mu)}$ and an additional source term $r$.

\subsection{Assumptions}\label{subsec:assumptions}

Our analysis relies on propagating regularity results from $\alpha$ and test functions to the solutions of the master equation, which we will introduce below.
Frequently, the regularity that we require is for a function to be Lipschitz continuous with Lipschitz continuous linear functional derivatives.
Thus, we combine these assumptions and make the definition that, for $F \from S_d \to E$ with $E \subset \RR^d$, $F$ lies in \textbf{(D-Lip)} when all components of $F$ are Lipschitz continuous and differentiable with Lipschitz continuous derivatives.

We introduce two ergodicity assumptions, which will be central to our analysis.
The first of these is quite standard, and in particular, allows for the existence of continuous functions on $S_d$ defined by the solutions of the Kolmogorov equation.

\begin{asmp}
	\textbf{(Erg)}	
	We say that the transition rates function $\alpha$ satisfy \textbf{(Erg)} if there exist constants $\lambda, c_1 > 0$ such that for all $\mu, \hat{\mu} \in S_d$ and $t \geq 0$, we have
	\begin{equation}\label{eq:convergence}
		\lvert m(t; \mu) - m(t; \hat{\mu}) \rvert \leq c_1 e^{-\lambda t} \lvert \mu - \hat{\mu} \rvert,
	\end{equation}
	which we call the \textit{exponential stability} of the Kolmogorov equation, and further we have that $m(t; \mu)$ is differentiable with respect to $\mu$ for all $t \geq 0$ with derivative satisfying
	\begin{equation}\label{eq:derLip}
		\left\lvert \frac{\delta m}{\delta \frm}(t, \mu, z) - \frac{\delta m}{\delta \frm}(t, \hat{\mu}, z) \right\rvert \leq k_1 \lvert \mu - \hat{\mu} \rvert,
	\end{equation}
    for a positive constant $k_1$.
\end{asmp}

\begin{rk}\label{rk:ergNote}
	This assumption implies that if $\alpha$ is in \textbf{(D-Lip)}, there exists a unique stationary solution $\nu_\infty$ of \eqref{eq:kolmogorov} and it is exponentially stable, i.e. for the constants $\lambda > 0$ and $c_1 > 0$ as above, for all $\mu \in S_d$ and $t \geq 0$, we have
	\begin{equation*}
		\lvert m(t; \mu) - \nu_\infty \rvert \leq c_1 e^{-\lambda t} \lvert \mu - \nu_\infty \rvert.
	\end{equation*}

	To show the uniqueness of the stationary measure, we suppose that $\mu_1$ and $\mu_2$ are both stationary measures. 
	We then have that $m(t; \mu_1) = \mu_1$, and $m(t; \mu_2) = \mu_2$ for all $t \geq 0$.
	Then, by \eqref{eq:convergence}, we must have 
	\begin{equation*}
		\lvert \mu_1 - \mu_2 \rvert \leq 2c_1 e^{-\lambda t},
	\end{equation*} for all $t \geq 0$, which clearly implies that $\mu_1 = \mu_2$, so any stationary measure is unique.
	To prove existence, we fix $\mu \in S_d$ and set $\mu_1 = m(s; \mu)$ and $\mu_2 = \mu$ in \eqref{eq:convergence}, which yields
	\begin{equation*}
		\lvert m(t+s; \mu) - m(t; \mu) \rvert \leq 2c_1 e^{-\lambda t},
	\end{equation*}
    since $m(t+s; \mu) = m(s; m(t; \mu))$. So $m(t; \mu)$ must converge to some limit as $t \to \infty$, which we denote by $\nu_\infty$.
	As, by lying in \textbf{(D-Lip)}, $\alpha$ is Lipschitz as a function of $m$, we have by \cite[XIV.3.2]{lang2012real}, that the map $\mu \mapsto m(t; \mu)$ is continuous.
	Thus, rewriting $m(t+s; \mu) = m(s; m(t; \mu))$ and taking $t\to\infty$ we see $m(s; \nu_\infty) = \nu_\infty$, and so $\nu_\infty$ is a stationary measure.
	Finally, taking $\mu_1 = \nu_\infty$ in \eqref{eq:convergence}, which by the above implies $m(t; \mu_1) = \nu_\infty$, we obtain that all $m(t, \mu)$ converge to $\nu_\infty$ exponentially fast.
\end{rk}

The nonlinearity of the Kolmogorov equation makes it difficult to directly verify \textbf{(Erg)}, so we also introduce an assumption that is easier to verify.

\begin{asmp}\textbf{(Lin-Erg)}.
	We say that $\alpha$ satisfies \textbf{(Lin-Erg)} if there exist $\lambda > 0$ and $c_2 \geq 0$, such that for any $\mu \in S_d$, $q_0 \in \RR^d_0$, and $r \from [0, \infty) \to \RR^d_0$ bounded and measurable, the solution $q(t)$ of the Cauchy problem \textbf{(Linear-$[\mu, q_0, r]$)} satisfies
	\begin{equation*}
		\lvert q(t) \rvert \leq c_2 \left[ e^{-\lambda t} \lvert q_0 \rvert + \int_0^t e^{-\lambda(t-s)} \lvert r(s) \rvert ds \right].
	\end{equation*}

\end{asmp}

We will show in \cref{prop:mConv} that \textbf{(Lin-Erg)} implies \textbf{(Erg)}.
Effectively, the relationship between \textbf{(Erg)} and \textbf{(Lin-Erg)} consists of replacing the Kolmogorov equation with all possible tangent equations.
Indeed, we will see that certain solutions of the linearized Kolmogorov equation represent functional derivatives of solutions to the Kolmogorov equation, and we will then integrate these solutions with respect to the initial condition to pass from \textbf{(Lin-Erg)} to \textbf{(Erg)}.

In contrast to linear Markov chains, even when the transition rates of the processes $X^{i,N}_t$ are bounded away from zero, this is insufficient to guarantee \textbf{(Erg)}.
We illustrate this with the following example, similar to that of Neumann \cite{Neumann2022}:

\begin{ex}\label{ex:nonErg}
	Taking $d = 2$, let
	\begin{equation*}
		\alpha(\mu) = 
		\begin{pmatrix}
			-(\mu_1^2+\mu_1+1) & \mu_1^2+\mu_1+1 \\
			31\mu_1^2-18\mu_1+3 & -(31\mu_1^2-18\mu_1+3)
		\end{pmatrix}.	
	\end{equation*}
	Notice that the transition rates are bounded in the interval $[1, 16]$.
	The system defined by this $\alpha$ does not satisfy \textbf{(Erg)}.
    This can be seen by the fact that $\Bar{\mu}_0 = (0.25, 0.75)$, $\Bar{\mu}_1 = (0.5, 0.5)$, and $\Bar{\mu}_2 = (0.75, 0.25)$ are all invariant under the Kolmogorov equation.
	By \cref{prop:mConv} this also implies that $\alpha$ does not satisfy \textbf{(Lin-Erg)}.
\end{ex}

Additionally, even when there is a unique, asymptotically stable stationary distribution, the convergence to this distribution may not be exponential, as in the following example:

\begin{ex}\label{ex:slowConv}
	Let $d = 2$ and let
	\begin{equation*}
		\alpha(\mu) = 
			\begin{pmatrix}
				-(2\mu_1^2+\mu_1+1) & 2\mu_1^2 + \mu_1 + 1 \\
				30\mu_1^2 - 19\mu_1 + 4 & -(30\mu_1^2 - 19\mu_1 + 4)
			\end{pmatrix}.
	\end{equation*}
	Here, there is a unique stationary distribution $\nu_\infty = (0.5, 0.5)$, and the solution to the Kolmogorov equation has the form
	\begin{equation*}
	m_1(t; \mu) = \frac{1}{2} + \frac{\operatorname{sgn}(\mu_1-\tfrac{1}{2})}{2 \sqrt{\tfrac{1}{(1-2\mu_1)^2} + 16t}},
	\end{equation*}
	for $\mu_1 \neq \tfrac{1}{2}$.
	It is clear that all solutions converge to $\nu_\infty$, but the convergence occurs at a rate of $1/\sqrt{t}$, rather than exponentially.
\end{ex}

Indeed, the behavior of a nonlinear Markov chain can be arbitrarily complex, and systems with $d \geq 4$ can even exhibit chaotic behavior, as shown in the example below:

\begin{ex}\label{ex:chaos}
	Letting $d = 4$, the mean field system defined by
	\begin{align*}
		&\alpha(\mu) = \\
		&\resizebox{.99\textwidth}{!}{
		$\begin{pmatrix}
			-(a+\rho+\sigma) - b \mu_2 - \frac{a(1+\beta+a)}{b \mu_1} & a + \rho + \frac{a}{b \mu_1} & b \mu_2 + \frac{a(\beta+a)}{b \mu_1} & \sigma \\
			\sigma & -(1 + \ell + \sigma) - \frac{a(\rho+a)+b^2 \mu_1 \mu_3}{b \mu_2} & \ell & 1 + \frac{a(\rho+a)+b^2 \mu_1 \mu_2}{b \mu_2} \\
			\ell & a & -(\beta + a + \ell) - \frac{a(\mu_1+\mu_2)}{\mu_3} & \beta + \frac{a(\mu_1 + \mu_2)}{\mu_3} \\
			\frac{1}{\mu_4}\left( \mu_1(a + \rho + b \mu_2) + \frac{a(1 + \beta + a)}{b} \right) & \frac{\sigma \mu_2}{\mu_4} & \frac{a \mu_3}{\mu_4} & -\frac{1}{\mu_4}\left( \mu_1(a + \rho + b \mu_2) + \frac{a(1 + \beta + a)}{b} + \sigma \mu_2 + c \mu_3 \right)
		\end{pmatrix}$}
	\end{align*}
    with $\sigma=10$, $\beta=8/3$, $\rho=28$, $a=35$, $b=200$, and $\ell=0.1$  has $\alpha$ in \textbf{(D-Lip)} with transition rates strictly bounded away from zero.
	The first three components of this system are a shifted and scaled version of the Lorenz system:
	\begin{align*}
		\frac{dx}{dt} &= -\sigma x + \sigma y, \\
		\frac{dy}{dt} &= \rho x - y - xz, \\
		\frac{dz}{dt} &= -\beta z + xy,
	\end{align*}
	which is known to exhibit chaotic behavior with the given parameters.

\end{ex}

\subsection{Main Result and the Master Equation}\label{subsec:mainResult}

We are interested in bounding quantities of the form 
\begin{equation*}
	\Big\lvert \EE\left[ \Phi(\mu_t^N) \right] - \Phi(\LL(Y_t)) \Big\rvert,
\end{equation*}
where $\Phi \from S_d \to \RR$ is in \textbf{(D-Lip)} acting as a test function, and $\mu_t^N$ is as in \eqref{eq:markov}.
More precisely, our main result is the following:
\begin{thm}\label{thm:mainRes}
    Assume that $\alpha$ is in \textbf{(D-Lip)} and satisfies \textbf{(Erg)}, and that for all $N \geq 1$, we have that the $\xi^i$ in \eqref{eq:markov} are independent with distribution $\mu_0: = \LL(Y_0)$. 
	Then, there exists a constant $C$, independent of $\mu_0$, such that for any $N \geq 1$ and $\Phi$ in \textbf{(D-Lip)}, we have
	\begin{equation*}
		\sup_{t \geq 0} \Big\lvert \EE\left[ \Phi(\mu_t^N) \right] - \Phi(\LL(Y_t)) \Big\rvert \leq \frac{C}{N}.
	\end{equation*}
\end{thm}

We now briefly describe the proof of the of this result, including the key intermediate results.
To approach the problem, we investigate the function $\UU \from [0, \infty) \times S_d \to \RR$ defined by
\begin{equation}\label{eq:Udef}
	\UU(t, \mu) \defeq \Phi(m(t; \mu)).
\end{equation}

We show that $\UU$ has the following properties:
\begin{prop}\label{prop:master}
\leavevmode
	\begin{enumerate}
		\item 
			Assume that $\alpha$ is in \textbf{(D-Lip)} and satisfies \textbf{(Erg)}.
            For $\Phi$ in \textbf{(D-Lip)}, we have that for any $t \geq 0$, $\mu \in S_d$, and $z \in [d]$, $\frac{\delta \UU}{\delta \frm}(t, \mu, z)$ exists and is Lipschitz continuous in $\mu$, uniformly in time.
			
		\item
			Assuming that $\Phi$ is differentiable, $\UU$ satisfies the following PDE, which we call the {\rm master equation}:
			\begin{equation}\label{eq:master}
				\frac{\partial \UU}{\partial t}(t, \mu) = \sum_{z \in [d]} \frac{\delta \UU}{\delta \frm}(t, \mu, z) (\mu \alpha(\mu))_z = D^\frm_1 \UU(t, \mu) (\mu \alpha(\mu))^T,
			\end{equation}
			with the initial condition $\UU(0, \mu) = \Phi(\mu)$.
	\end{enumerate}
\end{prop}

To proceed in our proof of the main result, we follow \cite{del-tse2021} and make the following decomposition:
\begin{equation}\label{eq:mainResDecomp}
	\Phi(\mu_t^N) - \Phi(\LL(Y_t)) = \Big( \UU(0, \mu_t^N) - \UU(t, \mu_0^N) \Big) + \Big( \UU(t, \mu_0^N) - \UU(t, \mu_0) \Big).
\end{equation}
We will then bound the expectation of each of the two terms on the right-hand side separately.

For the first term on the right-hand side of \ref{eq:mainResDecomp}, we define the function $U_t \from [0,t] \times [d]^N \to \RR$ for any $t \geq 0$ by
\begin{equation}\label{eq:finDimUdef}
	U_t(s, x_1, \dots, x_N) \defeq \UU\Big(t-s, \frac{1}{N} \sum_{i \in [N]} \delta_{x_i}\Big).
\end{equation}
This allows us to rewrite the first term in \eqref{eq:mainResDecomp} as
\begin{equation*}
	U_t(t, X_t^1, \dots, X_t^N) - U_t(0, X_0^1, \dots, X_0^N),
\end{equation*}
and we have the following result:
\begin{prop}\label{prop:firstTermBound}
	 Assuming that $\alpha$ is in \textbf{(D-Lip)} and satisfies \textbf{(Erg)}, and that $\Phi$ is in \textbf{(D-Lip)}, for all $t \geq 0$, we have that
	\begin{equation*}
		\Big\lvert \EE\left[ U_t(t, \mathbf{X}^N_t) - U_t(0, \mathbf{X}^N_0) \right] \Big\rvert \leq \frac{C}{N},
	\end{equation*}
	where $C$ is independent of $t$, $N$ and the distribution of $\mathbf{X}^N_0$.
\end{prop}
We prove this result by applying the master equation to a form of It\^o's formula, as well as using the regularity results proved for the solutions of the Kolmogorov equation.

Likewise, for the second term on the right-hand side of \eqref{eq:mainResDecomp}, we obtain that
\begin{prop}\label{prop:secondTermBound}
	Under the assumptions of \cref{thm:mainRes}, we have that
	\begin{equation*}
		\Big\lvert \EE\left[ \UU(t, \mu_0^N)\right] - \UU(t, \mu_0)  \Big\rvert \leq \frac{C}{N},
	\end{equation*}
	where $C$ is again independent of $t$, $N$ and $\mu_0:=\text{Law}(Y_0)$.
\end{prop}
This result is more straightforward and follows from \cref{prop:master} Part (1), and the general properties of the functional derivative.

Clearly, \cref{prop:firstTermBound} and \cref{prop:secondTermBound} together imply \cref{thm:mainRes}, and so we will work towards proving these two results in \Cref{sec:mainResProof}.

\section{Proof of \texorpdfstring{\cref{thm:mainRes}}{the Main Result}}\label{sec:mainResProof}

We begin by showing our main result under the assumption \textbf{(Erg)}.
First, we prove \cref{prop:master}, showing that $\UU$ is in \textbf{(D-Lip)}, and that $\UU$ solves the master equation.
We will use the regularity of $\UU$ several times in the proofs of \cref{prop:firstTermBound} and \cref{prop:secondTermBound}.
We then prove \cref{prop:firstTermBound} through a form of It\^o's formula, to which we can apply the master equation.
Finally, we complete the proof of our main result by directly proving \cref{prop:secondTermBound}.

\subsection{Proof of \texorpdfstring{\cref{prop:master}}{the properties of the master equation}}

\begin{proof}[Proof of \cref{prop:master}]
\leavevmode

{\bf Proof of part (1):} 
		To show that $\UU$ is differentiable if $\Phi$ is, we apply the chain rule for functional derivatives (\cref{prop:chainRule}), and see that for $\mu \in S_d$ and $z \in [d]$, the derivative of $\UU$ is given by
		\begin{equation}\label{eq:UDer}
			\frac{\delta \UU}{\delta \frm}(t, \mu, z) = \frac{\delta \Phi}{\delta \frm}(m(t; \mu), \cdot) \left( \frac{\delta m}{\delta \frm}(t, \mu, z) \right)^T = D^\frm_1 \Phi(m(t; \mu)) \left( \frac{\delta m}{\delta \frm}(t, \mu, z) \right)^T.
		\end{equation}
		The second equality follows from \eqref{eq:derRel1} and the fact that $\frac{\delta m}{\delta \frm}(t, \mu, z) \in \RR_0^d$ for all $t, \mu, z$.

		We next show that $\frac{\delta \UU}{\delta \frm}$ is Lipschitz continuous in $\mu$ uniformly in time.
		Using the above expression for derivatives of $\UU$, we have that
		\begin{align*}
			\left\lvert \frac{\delta \UU}{\delta \frm}(t, \mu, z) - \frac{\delta \UU}{\delta \frm}(t, \nu, z) \right\rvert &= \left\lvert \frac{\delta \Phi}{\delta \frm}(m(t; \mu), \cdot) \left( \frac{\delta m}{\delta \frm}(t, \mu, z) \right)^T \right. \\
			& \qquad \left. - \frac{\delta \Phi}{\delta \frm}(m(t; \nu), \cdot) \left( \frac{\delta m}{\delta \frm}(t, \nu, z) \right)^T \right\rvert \\
			&= \left\lvert \left( \frac{\delta \Phi}{\delta \frm}(m(t; \mu), \cdot) - \frac{\delta \Phi}{\delta \frm}(m(t; \nu), \cdot) \right) \left( \frac{\delta m}{\delta \frm}(t, \mu, z) \right)^T \right. \\
			& \qquad \left. + \frac{\delta \Phi}{\delta \frm}(t, \nu, \cdot) \left( \frac{\delta m}{\delta \frm}(t, \mu, z) - \frac{\delta m}{\delta \frm}(t, \nu, z) \right)^T \right\rvert.
		\end{align*}
        We recall that we require $\Phi$ and $\frac{\delta \Phi}{\delta \frm}$ to be Lipschitz continuous.
		Then, denoting the Lipschitz constant of $\frac{\delta \Phi}{\delta \frm}$ by $K$, we have that 
		\begin{equation*}
			\left\lvert \left( \frac{\delta \Phi}{\delta \frm}(m(t; \mu), \cdot) - \frac{\delta \Phi}{\delta \frm}(m(t; \nu), \cdot) \right) \left( \frac{\delta m}{\delta \frm}(t, \mu, z) \right)^T \right\rvert \leq K \lvert m(t; \mu) - m(t; \nu) \rvert \leq c_1 K \lvert \mu - \nu \rvert,
		\end{equation*}
		where the last inequality follows from the uniform Lipschitz continuity of $m(t; \mu)$ included in \textbf{(Erg)}.
		Then, denoting the Lipschitz constant of $\Phi$ by $C$, we also have that 
		\begin{equation*}
			\left\lvert \frac{\delta \Phi}{\delta \frm}(t, \nu, \cdot) \left( \frac{\delta m}{\delta \frm}(t, \mu, z) - \frac{\delta m}{\delta \frm}(t, \nu, z) \right)^T \right\rvert \leq C k_1 \lvert \mu - \nu \rvert,
		\end{equation*}
        where we have used the uniform in time Lipschitz continuity of $\frac{\delta m}{\delta \frm}$ assumed in \textbf{(Erg)}.
		By the triangle inequality, this yields the uniform Lipschitz continuity of $\frac{\delta \UU}{\delta \frm}$.

{\bf Proof of part (2):} 	
	The second equality in \eqref{eq:master} follows from \eqref{eq:derRel1} and the fact that $\mu \alpha(\mu) \in \RR^d_0$ for all $\mu \in S_d$. Hence, the rest of the proof is dedicated to the first equality.
	 
		From the semi-group property of $m(t; \mu)$, we have that 
		 \begin{align}\label{eq:semi-group}
		 m(t-s; m(s; \mu)) = m(t; \mu),\qquad t\ge s\ge0.
		 \end{align}
		As $\UU(t, \mu) = \Phi(m(t; \mu))$, we thus obtain that $\UU(t-s, m(s; \mu)) = \UU(t, \mu)$, and so is independent of $s$.
		Expanding $0=\frac{d}{ds}\UU(t-s, m(s; \mu)) $, making use of the chain rule in \Cref{prop:chainRule}, and using \eqref{eq:semi-group}, we get
	\begin{equation}\label{eq:phiexp}
			0 = \frac{d}{ds} \Phi(m(t-s; m(s; \mu))) = \sum_{z \in [d]} \frac{\delta \Phi}{\delta \frm} (m(t; \mu), z) \frac{d}{ds} m_z(t-s; m(s; \mu)).
		\end{equation}
		Then, by the chain rule and \eqref{eq:kolmogorov},
        we have
		\begin{align*}
			\frac{dm_z}{ds}(t-s; m(s; \mu)) &= - \frac{\partial m_z(t-s; m(s;\mu))}{\partial t} + \sum_{y \in [d]} \frac{\delta m_z}{\delta \frm} (t-s, m(s; \mu), y) \frac{dm_y}{ds} (s; \mu) \\
										   &= \sum_{y \in [d]} \left[ -m_y(t; \mu)\alpha_{yz}(m(t; \mu)) + \frac{\delta m_z}{\delta \frm} (t-s, m(s; \mu), y) (m(s; \mu) \alpha_y(m(s; \mu))) \right].
		\end{align*}
		As this must hold for all $s$, we can take $s = 0$ and combine the above with \eqref{eq:phiexp} to get
		\begin{equation*}
			0 = \sum_{z \in [d]} \frac{\delta \Phi}{\delta \frm} (m(t; \mu), z) \left( - m(t; \mu) \alpha(m(t; \mu)) + \frac{\delta m}{\delta \frm} (t, \mu, \cdot) (\mu \alpha(\mu))^T \right)_z.
		\end{equation*}
		Rearranging, we have
		\begin{equation}\label{eq:phifinal}
			\frac{\delta \Phi}{\delta \frm} (m(t; \mu), \cdot) \cdot \left( m(t; \mu) \alpha(m(t; \mu)) \right) = \left( \frac{\delta \Phi}{\delta \frm} (m(t; \mu), \cdot) \frac{\delta m}{\delta \frm} (t, \mu, \cdot) \right) (\mu \alpha(\mu))^T.
		\end{equation}
		Expressing $\UU$ using \eqref{eq:Udef}, and applying the chain rule from \cref{prop:chainRule}, we have
		\begin{equation}
			\frac{\partial \UU}{\partial t} (t, \mu) = \frac{\delta \Phi}{\delta \frm} (m(t; \mu), \cdot) \cdot \left( m(t; \mu) \alpha(m(t; \mu)) \right),
		\end{equation}
		and
		\begin{equation*}
			\frac{\delta \UU}{\delta \frm} (t, \mu, z) = \frac{\delta \Phi}{\delta \frm} (m(t; \mu), \cdot) \cdot \frac{\delta m}{\delta \frm} (t, \mu, z).
		\end{equation*}
		Substituting these into \eqref{eq:phifinal} yields the master equation as desired.
\end{proof}

\subsection{Proof of \texorpdfstring{\cref{prop:firstTermBound}}{the bound on the first term}}

We begin by showing a form of It\^o's formula for $U_t$.

\begin{lem}\label{lem:ito}
	Assuming that $\UU$ is differentiable, such as is the case if $\Phi$ is differentiable, and letting $\mathbf{X}^N_s \defeq \Big( X^{i,N}_s \Big)_{i \in [N]}$ be a vector of processes as in \eqref{eq:markov}, we have
	\begin{equation*}
		\EE\left[ U_T(t, \mathbf{X}^N_t) - U_T(0, \mathbf{X}^N_0) \right] = \EE\left[ \int_0^t \partial_s U_T(s, \mathbf{X}^N_s) ds + \sum_{i \in [N]} \int_0^t \Delta^i U_T(s, \mathbf{X}^N_{s^-}) \alpha(X^{i,N}_{s^-}, \mu^N_{s^-}) ds \right],
	\end{equation*}
     where we recall that we denote the $x$th row of $\alpha(\mu)$ by $\alpha(x, \mu)$.
\end{lem}

\begin{proof}
	We first note by the standard form of It\^o's formula for jump processes that
	\begin{equation*}
		U_T(t, \mathbf{X}^N_t) - U_T(0, \mathbf{X}^N_0) = \int_0^t \partial_s U_T(s, \mathbf{X}^N_s) ds + \sum_{s \leq t, \Delta \mathbf{X}^N_s \neq 0} U_T(s, \mathbf{X}^N_s) - U_T(s, \mathbf{X}^N_{s^-}).
	\end{equation*}
	As the jump times of the $X^{i,N}$ are almost surely distinct, this is almost surely equivalent to
	\begin{equation*}
		U_T(t, \mathbf{X}^N_t) - U_T(0, \mathbf{X}^N_0) = \int_0^t \partial_s U_T(s, \mathbf{X}^N_s) ds + \sum_{i \in [N]} \sum_{s \leq t, \Delta X^{i,N}_s \neq 0} U_T(s, [\mathbf{X}^{-i, N}_{s^-}; X^{i,N}_s]) - U_T(s, \mathbf{X}^N_{s^-}).
	\end{equation*}
	Taking expectations and applying \eqref{eq:markov}, we obtain
	\begin{multline*}
		\EE\left[ U_T(t, \mathbf{X}^N_t) - U_T(0, \mathbf{X}^N_0) \right] = \EE\bigg[ \int_0^t \partial_s U_T(s, \mathbf{X}^N_s) ds + \\ 
		+ \sum_{i \in [N]} \int_0^t \int_{\AAA^d} U_T(s, [\mathbf{X}^{-i, N}_{s^-}; X^{i,N}_{s^-} + \sum_{y \in [d]} (y - X_{s^-}) \mathbf{1}_{0 < \xi_y < \alpha_y(X^{i,N}_{s^-}, \mu^N_{s^-})}] ) - U_T(s, \mathbf{X}^N_{s^-}) \nu(d\xi) ds \bigg].
	\end{multline*}
	We can rewrite the inside most integral as
	\begin{equation*}
		\int_{\AAA^d} \sum_{z \in [d]} \left( U_T(s, [\mathbf{X}^{-i, N}_{s^-}; z]) - U_T(s, \mathbf{X}^N_{s^-}) \right) \mathbf{1}_{0 < \xi_z < \alpha_z(X^{i,N}_{s^-}, \mu^N_{s^-})} \nu(d\xi),
	\end{equation*}
	and as the terms involving $U_T$ are independent of $\xi$, this is equal to
	\begin{equation*}
		\sum_{z \in [d]} \left( U_T(s, [\mathbf{X}^{-i, N}_{s^-}; z]) - U_T(s, \mathbf{X}^N_{s^-}) \right) \alpha_z(X^{i,N}_{s^-}, \mu^N_{s^-}).
	\end{equation*}
	Substituting this back into the expectation, we obtain the desired result.
\end{proof}

Applying the master equation to the above form of It\^o's formula will allow us to progress, but in order to do so, we need to express the $\Delta^i U_t$ terms as expressions in the derivatives of $\UU$.

\begin{lem}
	For $U_t$ defined as in \eqref{eq:finDimUdef}, under the assumption of \cref{lem:ito}, we have that
	\begin{equation*}
		(\Delta^i U_t(s, \mathbf{X}^N_s))_z = \frac{1}{N} D^\frm_{X^{i,N}_s z}\UU(t-s, \mu^N_s) + \frac{1}{N} \tau_z(t-s, \mathbf{X_s}, i),
	\end{equation*}
	where
	\begin{equation*}
		\tau_z(t-s, \mathbf{X_s}, i) \defeq \int_0^1 \Big[D^\frm_{X^{i,N}_s z} \UU\Big(t-s, \mu^N_s + \tfrac{\zeta}{N}(\delta_z - \delta_{X^{i,N}_s}) \Big) - D^\frm_{X^{i,N}_s z} \UU(t-s, \mu^N_s)\Big] d\zeta.
	\end{equation*}
\end{lem}

\begin{proof}
	By applying \eqref{eq:differenceIntegral}, we obtain that
	\begin{equation*}
		(\Delta^i U_t(s, \mathbf{X}^N_s))_z = \frac{1}{N} \int_0^1 \Big[\sum_{y \in [d]} \frac{\delta \UU}{\delta \frm}\Big(t-s, \mu^N_s + \tfrac{\zeta}{N}(\delta_z - \delta_{X^{i,N}_s}), y\Big) (\delta_z - \delta_{X^{i,N}_s})_y \Big]d\zeta.
	\end{equation*}
	Simplifying, this is equal to
	\begin{equation*}
		\frac{1}{N} \int_0^1 \Big[\frac{\delta \UU}{\delta \frm}\Big(t-s, \mu^N_s + \tfrac{\zeta}{N}(\delta_z - \delta_{X^{i,N}_s}), z\Big) - \frac{\delta \UU}{\delta \frm}\Big(t-s, \mu^N_s + \tfrac{\zeta}{N}(\delta_z - \delta_{X^{i,N}_s}), X^{i,N}_s\Big) \Big]d\zeta,
	\end{equation*}
	which, by \eqref{eq:derRel1}, is equal to
	\begin{equation*}
		\frac{1}{N} \int_0^1 \Big[D^\frm_{X^{i,N}_s z} \UU\Big(t-s, \mu^N_s + \tfrac{\zeta}{N}(\delta_z - \delta_{X^{i,N}_s}) \Big) \Big]d\zeta.
	\end{equation*}
	We add and subtract $D^\frm_{X^{i,N}_s z} \UU(t-s, \mu^N_s)$ to get the desired result.
\end{proof}

Now, we may combine the two preceding lemmas with the master equation, obtaining:

\begin{prop}\label{prop:firstTermExp}
	Again assuming that $\UU$ is differentiable, we have the expression
	\begin{align}\label{eq:mainResP1}\begin{split}
		\EE\left[ U_T(t, \mathbf{X}^N_t) - U_T(0, \mathbf{X}^N_0) \right] &= \EE\bigg[ \int_0^t \frac{1}{N} \sum_{i \in [N]} \tau(T-s, \mathbf{X}^N_{s^-}, i) \alpha(X^{i, N}_{s^-}, \mu^N_{s^-})^T ds \bigg].
    \end{split}
	\end{align}

\end{prop}

\begin{proof}
	We apply the preceding lemma to It\^o's formula to get
	\begin{align*}
		&\EE\left[ U_T(t, \mathbf{X}^N_t) - U_T(0, \mathbf{X}^N_0) \right] = \\ 
		&\qquad = \EE\bigg[ \int_0^t \partial_s U_T(s, \mathbf{X}^N_s) ds \\
		&\qquad \qquad + \sum_{i \in [N]} \int_0^t \frac{1}{N} \Big( D^\frm_{X^{i,N}_{s^-}}\UU(T-s, \mu^N_{s^-}) + \tau(T-s, \mathbf{X}^N_{s^-}, i) \Big) \alpha(X^{i,N}_{s^-}, \mu^N_{s^-})^T ds \bigg].
	\end{align*}
	We then apply the master equation to the first integral to get that this is equal to
	\begin{align*}
		&\EE\bigg[ \int_0^t -D^\frm_1 \UU(T-s, \mu^N_s)(\mu^N_s \alpha(\mu^N_{s}) )^T \\
		&\qquad + \frac{1}{N} \sum_{i \in [N]}\Big( D^\frm_{X^{i,N}_{s^-}}\UU(T-s, \mu^N_{s^-}) + \tau(T-s, \mathbf{X}^N_{s^-}, i) \Big) \alpha(X^i_{s^-}, \mu^N_{s^-})^T ds \bigg].
	\end{align*}
     As we have that $D^\frm_{X^{i,N}_{s^-} y}\UU(T-s, \mu^N_{s^-}) = D^\frm_{X^{i,N}_{s^-} 1}\UU(T-s, \mu^N_{s^-}) + D^\frm_{1y}\UU(T-s, \mu^N_{s^-})$, we get that
    \begin{align*}
        D^\frm_{X^{i,N}_{s^-}}\UU(T-s, \mu^N_{s^-}) \alpha(X^i_{s^-}, \mu^N_{s^-})^T &= \sum_{y \in [d]} D^\frm_{X^{i,N}_{s^-} 1}\UU(T-s, \mu^N_{s^-})\alpha_y(X^i_{s^-}, \mu^N_{s^-}) \\
        &\qquad + D^\frm_1\UU(T-s, \mu^N_{s^-}) \alpha(X^i_{s^-}, \mu^N_{s^-})^T.
    \end{align*}
	Using that $\alpha_x(\mu) \in \RR^d_0$ for all $x \in [d]$ and $\mu \in S_d$, the first sum vanishes, so 
    \begin{equation*}
        D^\frm_{X^{i,N}_{s^-}}\UU(T-s, \mu^N_{s^-}) \alpha(X^i_{s^-}, \mu^N_{s^-})^T = D^\frm_1\UU(T-s, \mu^N_{s^-}) \alpha(X^i_{s^-}, \mu^N_{s^-})^T.
    \end{equation*}
    
	Then, we note that
	\begin{equation*}
		\sum_{i \in [N]} \frac{1}{N} \alpha(X^{i,N}_{s^-}, \mu^N_{s^-}) = \mu^N_{s^-} \alpha(\mu^N_{s^-}).
	\end{equation*}
	Thus, we have
	\begin{align*}
		\EE\left[ U_T(t, \mathbf{X}^N_t) - U_T(0, \mathbf{X}^N_0) \right] &= \EE\bigg[ \int_0^t\Big\{ D^\frm_1 \UU(T-s, \mu^N_{s^-}) \Big(\mu^N_{s^-} \alpha(\mu^N_{s^-}) \Big)^T \\
		& \qquad\qquad\quad - D^\frm_1 \UU(T-s, \mu^N_s) \Big(\mu^N_s \alpha(\mu^N_s) \Big)^T \\ 
		& \qquad + \frac{1}{N} \sum_{i \in [N]} \tau(T-s, \mathbf{X}^N_{s^-}, i) \alpha(X^{i,N}_{s^-}, \mu^N_{s^-})\Big\} ds \bigg].
	\end{align*}
     Note that the integral of the first two terms is 0. This follows from the fact that the set of jump times of the process $\mu_s^N$ is almost surely discrete. Hence, we obtain \eqref{eq:mainResP1}.
    
\end{proof}

The above representation of the term under consideration for \cref{prop:firstTermBound} allows us to apply the regularity results we have obtained.
Indeed if we can show that the integrand in \eqref{eq:mainResP1} decays sufficiently quickly in $s$ and is of order $1/N$, we will have proven \cref{prop:firstTermBound}.

\begin{proof}[Proof of \cref{prop:firstTermBound}]
    
	We show that the integrand in \eqref{eq:mainResP1} decays sufficiently quickly.
	 Using the chain rule, we can rewrite $\tau_z(t-s, \mathbf{X}^N_{s^-}, i)$ as
	\begin{align*}
		&\int_0^1 \Big[\frac{\delta \Phi}{\delta \frm}\Big(m\big(t-s; \mu^N_{s^-} + \tfrac{\zeta}{N}(\delta_z - \delta_{X^{i,N}_{s^-}})\big), \cdot\Big) \Big( D^\frm_{X^{i,N}_{s^-}z}m\big(t-s; \mu^N_{s^-} + \tfrac{\zeta}{N}(\delta_z - \delta_{X^{i,N}_{s^-}}) \big) \Big)^T \\ 
		&\qquad - \frac{\delta \Phi}{\delta \frm}\big(m(t-s; \mu^N_{s^-}), \cdot\big) \Big( D^\frm_{X^{i,N}_{s^-}z}m(t-s, \mu^N_{s^-}) \Big)^T\Big] d\zeta.
	\end{align*}
	Again using \textbf{(Erg)} and that $\frac{\delta \Phi}{\delta \frm}$ is bounded and Lipschitz yields that there exists $C>0$, independent of $t,N$, and $\text{Law}(\boldsymbol{X}_0)$, such that,
	\begin{equation*}
		\Big\lvert \tau_z(t-s, \mathbf{X}^N_{s^-}, i) \Big\rvert \leq C k_1 e^{-\lambda (t-s)}.
	\end{equation*}
	Combining this with the boundedness of $\alpha$, we obtain the desired decay.

    To conclude, we need that $\tau(t-s, \mathbf{X}^N_{s^-}, i)$ is of order $1/N$, but this follows immediately from the Lipschitz continuity of $\frac{\delta \Phi}{\delta \frm}$ and $D^\frm m$, so we are done.

\end{proof}

\subsection{Proof of \texorpdfstring{\cref{prop:secondTermBound}}{the bound on the second term}}

\begin{proof}[Proof of \cref{prop:secondTermBound}]
	Recall the notation $\mu_0:=\text{Law}(Y_0)$. We first define
	\begin{equation*}
		\mu^N_{\zeta} \defeq \mu^N_0 + \zeta \left( \mu_0 - \mu^N_0 \right),
	\end{equation*}
	then, by \eqref{eq:differenceIntegral}, we have that
	\begin{equation*}
		 \UU(t, \mu_0) - \UU(t, \mu_0^N) = \int_0^1 \sum_{z \in [d]} \frac{\delta \UU}{\delta \frm}(t, \mu_{\zeta}^N, z) (\mu_0 - \mu_0^N)_z d\zeta.
	\end{equation*}
	Defining $\Tilde{\xi}_1$ to be a random variable with law $\mu_0$ independent of $\xi_1, \dots, \xi_N$, the initial conditions of the $N$ processes and taking expectations, we obtain
	\begin{equation*}
		\EE\left[ \UU(t, \mu_0) - \UU(t, \mu_0^N) \right] = \int_0^1 \EE\left[ \frac{\delta \UU}{\delta \frm}(t, \mu_{\zeta}^N, \Tilde{\xi}_1) - \frac{1}{N} \sum_{i \in N} \frac{\delta \UU}{\delta \frm}(t, \mu_{\zeta}^N, \xi_i) \right] d\zeta,
	\end{equation*}
	which, by our assumption for \cref{thm:mainRes} that the $\{\xi_i\}_{i \in [N]}$ are identically distributed according to $\mu_0$, is equal to
	\begin{equation*}
		\int_0^1 \EE\left[ \frac{\delta \UU}{\delta \frm}(t, \mu_{\zeta}^N, \Tilde{\xi}_1) - \frac{\delta \UU}{\delta \frm}(t, \mu_{\zeta}^N, \xi_1) \right] d\zeta.
	\end{equation*}
	Defining $\Tilde{\mu}_0^N \defeq \mu_0^N + \frac{1}{N}(\Tilde{\xi}_1 - \xi_1)$, and $\Tilde{\mu}_{\zeta}^N \defeq \Tilde{\mu}_0^N + \zeta (\mu_0 - \Tilde{\mu}_0^N)$, we see that
	\begin{equation*}
		\EE\left[ \frac{\delta \UU}{\delta \frm}(t, \mu_{\zeta}^N, \xi_1) \right] = \EE\left[ \frac{\delta \UU}{\delta \frm}(t, \Tilde{\mu}_{\zeta}^N, \Tilde{\xi}_1) \right],
	\end{equation*}
	and thus
	\begin{equation*}
		\EE\left[ \UU(t, \mu_0) - \UU(t, \mu_0^N) \right] = \int_0^1 \EE\left[ \frac{\delta \UU}{\delta \frm}(t, \mu_{\zeta}^N, \Tilde{\xi}_1) - \frac{\delta \UU}{\delta \frm}(t, \Tilde{\mu}_{\zeta}^N, \Tilde{\xi}_1) \right] d\zeta.
	\end{equation*}
	Taking norms, and using the uniform in time Lipschitz continuity of $\frac{\delta \UU}{\delta \frm}$, we get, for $C$ a constant independent of $t$, $N$, and $\mu_0$, that 
	\begin{equation*}
		\bigg\lvert \EE\left[ \UU(t, \mu_0) - \UU(t, \mu_0^N) \right] \bigg\rvert \leq \EE\left[ C \int_0^1 \lvert \Tilde{\mu}_\zeta^N - \mu_\zeta^N \rvert d\zeta \right].
	\end{equation*}
	It is straightforward to see that $\lvert \Tilde{\mu}_\zeta^N - \mu_\zeta^N \rvert \leq \frac{2}{N}$, and so, by modifying $C$, we finally obtain
	\begin{equation*}
		\bigg\lvert \EE\left[ \UU(t, \mu_0^N) - \UU(t, \mu_0) \right]  \bigg\rvert \leq \frac{C}{N},
	\end{equation*}
	as desired.
\end{proof}

\section{Exponential Stability from the Linearized System}\label{sec:linErgToErg}

In this section, we show that the exponential stability of the linearized Kolmogorov equation, \textbf{(Lin-Erg)}, implies the exponential stability of the original Kolmogorov equation, \textbf{(Erg)}.
We begin our analysis by proving several auxiliary results on the measure flow $m(t; \mu)$ and the solutions to the linearized Kolmogorov equation.
In particular, in \cref{lem:mDerivative}, we show that the derivative of $m(t; \mu)$ with respect to $\mu$ is realized by the solution to a linearized Cauchy problem.
This is crucial, as it provides the link between the bounds provided by \textbf{(Lin-Erg)} and the bounds we need to prove \textbf{(Erg)}.
Using the fundamental theorem of calculus for functional derivatives to apply this, we obtain \cref{prop:mConv}, which shows that $m(t; \mu)$ is exponentially stable, and in \textbf{(D-Lip)} with uniform Lipschitz constants in time.

Initially, we find it necessary to fix a time horizon and prove two regularity results with constants depending on this horizon.
Later, we will be able to remove this restriction under \textbf{(Lin-Erg)}.
For any $\mu, \nu \in S_d$, we define $m^{(1)}(t, \mu, \nu)$ to be the solution to the Cauchy problem \textbf{(Linear-$[\mu, \nu-\mu, 0]$)}.

\begin{lem}\label{lem:rhoBound}
	Assuming $\alpha$ is in \textbf{(D-Lip)}, then for a fixed $T > 0$, and for all $\mu, \hat{\mu} \in S_d$, and any $t \in [0,T]$, there exist positive constants $C_1(T)$ and $C_2(T)$, potentially depending on $T$ but independent of $t$, $\mu$, and $\hat{\mu}$, such that
	\begin{equation}\label{eq:mLip}
		\lvert m(t; \hat{\mu}) - m(t; \mu) \rvert \leq C_1 \lvert \hat{\mu} - \mu \rvert,
	\end{equation}
	and
	\begin{equation}\label{eq:rhoBound}
		\lvert m(t; \hat{\mu}) - m(t; \mu) - m^{(1)}(t, \mu, \hat{\mu}) \rvert \leq C_2 \lvert \hat{\mu} - \mu \rvert^2.
	\end{equation}
\end{lem}

\begin{proof}
	To prove \eqref{eq:mLip}, we apply a form of Gronwall's inequality to $\lvert m(t; \hat{\mu}) - m(t; \mu) \rvert$.
	By the fundamental theorem of calculus and the Lipschitz continuity of $\alpha$, then for some $C>0$, independent of $s$, $t$, $\mu$, and $\hat{\mu}$, we have that 
	\begin{equation*}
		\lvert m(t; \hat{\mu}) - m(t; \mu) \rvert \leq \lvert \hat{\mu} - \mu \rvert + \int_0^t C \lvert m(s; \hat{\mu}) - m(s; \mu) \rvert ds.
	\end{equation*}
	Thus, by the standard form of Gronwall's inequality, $\lvert m(t; \hat{\mu}) - m(t; \mu) \rvert \leq C_1(T) \lvert \hat{\mu} - \mu \rvert$, for some positive constant $C_1(T)$, potentially depending on $T$ but independent of $\mu,\hat\mu$, and $t$. 

	We next define $\rho(t) \defeq m(t; \hat{\mu}) - m(t; \mu) - m^{(1)}(t, \mu, \hat{\mu})$.
	Differentiating $\rho$ with respect to $t$, we obtain
	\begin{align*}
		\frac{d}{dt} \rho(t) &= m(t; \hat{\mu}) \alpha(m(t; \hat{\mu})) - m(t; \mu) \alpha(m(t; \mu)) - m^{(1)}(t, \mu, \hat{\mu}) \alpha(m(t; \mu)) -\\
		&\qquad - m(t; \mu) \bigg( \sum_{z \in [d]} \frac{\delta \alpha}{\delta \frm}(m(t; \mu), z) m^{(1)}_z(t, \mu, \hat{\mu}) \bigg).
	\end{align*}
	Rewriting, we see that this is
	\begin{equation*}
		\frac{d}{dt} \rho(t) = L_{m(t; \mu)} \rho(t) + B(t),
	\end{equation*}
	where
	\begin{align*}
		B(t) &\defeq (m(t; \hat{\mu}) - m(t; \mu)) \Big(\alpha(m(t; \hat{\mu})) - \alpha(m(t; \mu))\Big) + \\
			 &\qquad + m(t; \mu) \bigg( \alpha(m(t; \hat{\mu})) - \alpha(m(t; \mu)) - \sum_{z \in [d]} \frac{\delta \alpha}{\delta \frm}(m(t; \mu), z) \big(m_z(t; \hat{\mu}) - m_z(t; \mu)\big) \bigg).
	\end{align*}

    In the following, we use a constant $\bar C$, which is independent of $t$, $\mu$, and $\hat{\mu}$, and may change from line to line.
	By Taylor's theorem and the fact that $\alpha$ is in \textbf{(D-Lip)}, we know that
	\begin{equation*}
		\bigg\lvert \alpha_{xy}(m(t; \hat{\mu})) - \alpha_{xy}(m(t; \mu)) - \sum_{z \in [d]} \frac{\delta \alpha_{xy}}{\delta \frm}(m(t; \mu), z) (m(t; \hat{\mu}) - m(t; \mu)) \bigg\rvert \leq \bar C \lvert m(t; \hat{\mu}) - m(t; \mu) \rvert^2.
	\end{equation*}
	Additionally, as $\alpha$ is Lipschitz in $m$, we have that
	\begin{equation*}
		\lvert \alpha_{xy}(m(t; \hat{\mu})) - \alpha_{xy}(m(t; \mu)) \rvert \leq \bar C \lvert m(t; \hat{\mu}) - m(t; \mu) \rvert.
	\end{equation*}
	Combining these with our expression for $B(t)$, we get
	\begin{equation*}
		\lvert B(t) \rvert \leq \bar C \lvert m(t; \hat{\mu}) - m(t; \mu) \rvert^2.
	\end{equation*}
	Applying \eqref{eq:mLip}, and incorporating a factor of $\bar C$ into $C_1(T)$, we have that
	\begin{equation*}
		\lvert B(t) \rvert \leq C_1(T) \lvert \hat{\mu} - \mu \rvert^2.
	\end{equation*}
	We now use that $\rho$ is the solution to \textbf{(Linear-$[\mu, 0, B]$)} and apply Gronwall's inequality to obtain
	\begin{equation*}
		\lvert \rho(t) \rvert \leq C_1(T) \int_0^t \lvert B(s) \rvert ds \leq C_2(T) \lvert \hat{\mu} - \mu \rvert^2,
	\end{equation*}
	as desired.

\end{proof}

By this result, we may think of $m^{(1)}(t, \mu, \nu)$ as the derivative of $m(t; \mu)$ with respect to $\mu$ in the direction $\nu$.
This is made precise by the following lemma:

\begin{lem}\label{lem:mDerivative}
	For $\alpha$ in \textbf{(D-Lip)}, $t \geq 0$, $\mu \in S_d$, and $z \in [d]$, we have that
	\begin{equation*}
		\frac{\delta m}{\delta \frm}(t, \mu, z) = m^{(1)}(t, \mu, \delta_z).
	\end{equation*}
\end{lem}

\begin{proof}
	For any fixed $t$, we can take $T > t$ and apply \cref{lem:rhoBound} to get that 
	\begin{equation*}
		\lvert m(t; \hat{\mu}) - m(t; \mu) - m^{(1)}(t, \mu, \hat{\mu}) \rvert \leq C_2(T) \lvert \hat{\mu} - \mu \rvert^2.
	\end{equation*}
	Then, we have that
	\begin{align*}
		\frac{\delta m}{\delta \frm}(t, \mu, z) &= \lim_{\varepsilon \to 0} \frac{m(t; (1-\varepsilon)\mu + \varepsilon \delta_z) - m(t; \mu)}{\varepsilon} \\
		&= \lim_{\varepsilon \to 0} \frac{m^{(1)}(t, \mu, (1-\varepsilon)\mu + \varepsilon \delta_z) + o(\varepsilon^2)}{\varepsilon} \\
		&= m^{(1)}(t, \mu, \delta_z),
	\end{align*}
	where the last equality follows from the fact that, since \textbf{(Linear-$[\mu, q_0, 0]$)} is linear and homogeneous, we have
    \begin{equation*}
        m^{(1)}(t, \mu, (1-\varepsilon)\mu + \varepsilon \hat{\mu}) = \varepsilon m^{(1)}(t, \mu, \hat{\mu}).
    \end{equation*}
\end{proof}

\begin{rk}\label{rk:m1Decay}
	The assumption \textbf{(Lin-Erg)} ensures that the derivatives of $m(t; \mu)$ with respect to $\mu$ are exponentially decaying.
	More precisely, as $m^{(1)}(t, \mu, \nu)$ is the solution to \textbf{(Linear-$[\mu, \nu-\mu, 0]$)}, the statement of \textbf{(Lin-Erg)} immediately shows that
	\begin{equation*}
		\sup_{\mu, \nu \in S_d} \lvert m^{(1)}(t, \mu, \nu) \rvert \leq 2 c_2 e^{-\lambda t},
	\end{equation*}
	where $\lambda$ and $c_2$ are as in \textbf{(Lin-Erg)}.
\end{rk}

The previous results now allow us to justify that \textbf{(Lin-Erg)} is indeed an ergodicity condition, and in fact implies \textbf{(Erg)}.

\begin{prop}\label{prop:mConv}
	Under \textbf{(Lin-Erg)} and the assumption that $\alpha$ is in \textbf{(D-Lip)}, $\alpha$ satisfies \textbf{(Erg)}, i.e. solutions to \eqref{eq:kolmogorov} are exponentially stable as in \eqref{eq:convergence}, differentiable with respect to the initial condition, and the derivative is Lipschitz continuous in the initial condition uniformly in time as in \eqref{eq:derLip}. 
\end{prop}

\begin{proof}
	 As we know by \cref{lem:mDerivative} that $m^{(1)}(t, \mu, \delta_z) = \frac{\delta m}{\delta \frm}(t, \mu, z)$, we may use the fundamental theorem of calculus to see that, for any $\mu, \hat{\mu} \in S_d$,
	\begin{equation*}
		m(t; \mu) - m(t; \hat{\mu}) = \int_0^1 \sum_{z \in [d]} m^{(1)}(t, \tau \mu + (1-\tau) \hat{\mu}, \delta_z) (\hat{\mu} - \mu)_z \; d\tau.
	\end{equation*}
	Then, by \cref{rk:m1Decay}, we get
	\begin{equation}
		\lvert m(t; \mu) - m(t; \hat{\mu}) \rvert \leq 2 c_2 \lvert \mu - \hat{\mu} \rvert e^{-\lambda t}.
	\end{equation}
    
    By \cref{lem:mDerivative}, we know that $m(t; \mu)$ is differentiable in $\mu$, so it only remains to show that the derivative, $\frac{\delta m}{\delta \frm}(t, \mu, z)$ is Lipschitz continuous in $\mu$ uniformly in $t$.

	We show this by expressing a difference of derivatives as a solution to a linearized Cauchy problem.
	For any $\mu, \hat{\mu} \in S_d$, we have by \cref{lem:mDerivative} that
	\begin{equation*}
		\frac{d}{dt}\left( \frac{\delta m}{\delta \frm}(t, \mu, z) - \frac{\delta m}{\delta \frm}(t, \hat{\mu}, z) \right) = L_{m(t; \mu)}\frac{\delta m}{\delta \frm}(t, \mu, z) - L_{m(t; \hat{\mu})} \frac{\delta m}{\delta \frm}(t, \hat{\mu}, z),
	\end{equation*}
	Letting $\rho(t) \defeq \frac{\delta m}{\delta \frm}(t, \mu, z) - \frac{\delta m}{\delta \frm}(t, \hat{\mu}, z)$, we see that $\rho$ satisfies
	\begin{equation*}
		\frac{d}{dt} \rho(t) = L_{m(t; \mu)} \rho(t) + r(t), \quad \rho(0) = \hat{\mu} - \mu,
	\end{equation*}
	where
	\begin{align*}
		r(t) &= \left( L_{m(t; \mu)} - L_{m(t; \hat{\mu})} \right) \frac{\delta m}{\delta \frm}(t, \hat{\mu}, z) \\
			 &= \frac{\delta m}{\delta \frm}(t, \hat{\mu}, z) \big( \alpha(m(t; \mu)) - \alpha(m(t; \hat{\mu})) \big) \\
			 &\qquad + m(t; \mu) \left( \sum_{y \in [d]} \frac{\delta \alpha}{\delta \frm}(m(t; \mu), y) \left(\frac{\delta m}{\delta \frm}(t, \hat{\mu}, z)\right)_y \right) \\
			 &\qquad - m(t; \hat{\mu}) \left( \sum_{y \in [d]} \frac{\delta \alpha}{\delta \frm}(t, m(t; \hat{\mu}), y) \left(\frac{\delta m}{\delta \frm}(t, \hat{\mu}, z)\right)_y \right).
	\end{align*}
	It follows from the fact that $\alpha$ is in \textbf{(D-Lip)} and the boundedness of $\frac{\delta m}{\delta \frm}(t, \mu, z)$ that
	\begin{equation*}
		\lvert r(t) \rvert \leq C \lvert \mu - \hat{\mu} \rvert,
	\end{equation*}
	for some constant $C>0$ independent of $\mu$, $\hat{\mu}$, and $t$.
	Then by \textbf{(Lin-Erg)}, we have that
	\begin{equation*}
		\lvert \rho(t) \rvert \leq c \left( e^{-\lambda t} \lvert \mu - \hat{\mu} \rvert + \lvert \mu - \hat{\mu} \rvert \right),
	\end{equation*}
	thus showing that $\frac{\delta m}{\delta \frm}(t, \mu, z)$ is Lipschitz continuous in $\mu$ uniformly in $t$.

\end{proof}

\section{Applications to Nonlinear Markov Chains and Mean Field Games}\label{sec:applications}

\subsection{Nonlinear Markov Chains}\label{subsec:nonlinearMarkov}

The question of exponential stability for nonlinear Markov chains is significantly more challenging than for linear Markov chains.
In particular, for $d >2$, it is difficult to obtain an easily verifiable condition that ensures exponential stability.
We apply our results in \Cref{sec:linErgToErg} to obtain two conditions for the exponential ergodicity of a nonlinear Markov chain in the case that $\alpha$ is in \textbf{(D-Lip)}.

\begin{prop}
	For $\alpha$ in \textbf{(D-Lip)}, the nonlinear Markov chain with transition rates given by $\alpha$ satisfies \textbf{(Erg)} and, in particular, is exponentially ergodic if we have either of the following conditions:
	\begin{enumerate}
		\item For all $x \neq y \in [d]$, there is a uniform lower bound, $L$, for $\alpha_{xy}(\mu)$, and the Lipschitz constant of $\alpha$, which we denote by $K$, satisfies $K < L/d$.
		\item For all $x, y \in [d]$ and $\mu \in S_d$, $\alpha$ satisfies
			\begin{equation*}
				\sum_{z, w \in [d]} \mu_z D^\frm_{wx}\alpha_{zy}(\mu) \mu_w > -\alpha_{xy}(\mu),
			\end{equation*}
	\end{enumerate}
\end{prop}

\begin{proof}
	For a given $\alpha$, define
	\begin{equation*}
		A_{xy}(t, \mu) \defeq m(t, \mu) \cdot \frac{\delta \alpha_{\cdot,y}}{\delta \frm}(m(t; \mu), x) + \alpha_{xy}(m(t; \mu)).
	\end{equation*}
	Then the linearized Kolmogorov equation takes the form
	\begin{equation*}
		\frac{d}{dt} q(t) = q(t) A(t, \mu) + r(t).
	\end{equation*}
	As shown in \cite[Lemma 3.3]{coh-zel2022}, if $A(t, \mu)$ has off-diagonal entries that are strictly positive, then \textbf{(Lin-Erg)} holds.
	Thus, by applying \cref{prop:mConv}, we obtain that the nonlinear Markov chain is exponentially ergodic if $A(t, \mu)$ has strictly positive off-diagonal entries.
	We can rephrase this condition as
	\begin{equation}\label{eq:offDiagBound}
		\sum_{z \in [d]} \mu_z \frac{\delta \alpha_{zy}}{\delta \frm}(\mu, x) > -\alpha_{xy}(\mu),
	\end{equation}
	for all $x \neq y \in [d]$.
	As $\alpha(\mu) > L$ for all $x \neq y$ and $\mu \in S_d$, it clearly suffices to have
	\begin{equation*}
		\sum_{z \in [d]} \mu_z \frac{\delta \alpha_{zy}}{\delta \frm}(\mu, x) > -L,
	\end{equation*}
	and so by the definition of $K$, we obtain the first condition.
	For the second condition, we have by \eqref{eq:derRel2} that \eqref{eq:offDiagBound} is equivalent to
	\begin{equation*}
		\sum_{z, w \in [d]} \mu_z D^\frm_{wx}\alpha_{zy}(\mu) \mu_w > -\alpha_{xy}(\mu),
	\end{equation*}
	which is exactly the second condition.
\end{proof}

\subsection{Large Population Games and Mean Field Games}\label{subsec:MFG}

Large population games are challenging to analyze due to the complexity of interactions and the difficulty in finding exact equilibria. In some specific cases, such as finite horizon games, solutions are more tractable when using either open-loop or closed-loop feedback controls. In the open-loop case, each player observes and reacts to their own dynamics, whereas in the closed-loop feedback case, players' actions depend on the current states of all the players \cite{CardaliaguetDelarueLasryLions}. 

For ergodic games, where each player aims to minimize a long-run average cost, results are known only in certain scenarios. For example, the open-loop control case was studied in \cite{bar-pri}, while the closed-loop feedback case remains an open problem. Recent progress has been made in \cite{coh-zel2022}, where asymptotic (in the number of players) Nash equilibria were constructed using policies derived from the limiting model, the ergodic MFG. 

Two approaches to constructing these asymptotic equilibria are presented in \cite{coh-zel2022}:
\begin{itemize}
    \item \textbf{Independent Controls:} Players use controls independent of the others, constructed from an algebraic system of equations describing the mean field equilibrium. The equations correspond to the Hamilton--Jacobi--Bellman (HJB) equation and the stationary distribution of the limiting MFG.
    \item \textbf{Mass-Dependent Controls:} Players respond to the current mass distribution of others. These controls are derived from the master equation of the MFG (different from \eqref{eq:master}), which provides a comprehensive characterization of the MFG. The master equation of the MFG  is a partial differential equation that captures the evolution of the value function and the distribution of players over time, enabling the construction of more concentrated equilibria.
\end{itemize}

In the following, we describe the ergodic $N$-player game, the asymptotic regimes of Nash equilibria, some previously established results, and the contribution of the current paper in extending these findings. 

\subsubsection{The Ergodic \texorpdfstring{$N$}{N}-Player Game - Setup} To make the section more self contained, we briefly describe the game from \cite{coh-zel2022}. Consider an $N$-player game where each player $i$ selects a feedback control (a rate matrix), denoted by $\beta^i$. By symmetry among players, $\beta^i$ can be assumed to depend only on the private state of player $i$ at time $t$, $X_t^{i,N}\in[d]$, and the empirical distribution of all players, $\mu^N_t := \mu^N_{\boldsymbol{X}^N_t}$, where 
$\boldsymbol{X}^N_t = (X_t^{1,N}, \ldots, X_t^{N,N})$ is the state of the players at time $t$.

The cost incurred by player $i$, given the initial state vector $\boldsymbol{X}^N_0 = (X_0^{1,N}, \ldots, X_0^{N,N}) \in [d]^N$ and a strategy profile $\boldsymbol{\beta} = (\beta^{1,N}, \ldots, \beta^{N,N})$, is defined as:
\[
J^{i,N}_0(\boldsymbol{X}_0, \boldsymbol{\beta}) := \limsup_{T \to \infty} \frac{1}{T} \mathbb{E} \Bigg[\int_0^T \Big\{f(X^{i,N}_t, \beta^i(X^{i,N}_t, \mu_t^N)) + F(X^{i,N}_t, \mu_t^N)\Big\} \, dt\Bigg],
\]
where $f$ and $F$ represent the running costs for player $i$ based on their individual and collective strategies.

An {\it $\epsilon$-Nash equilibrium} is a strategy profile in which no player can improve their cost by more than $\epsilon$ by unilaterally deviating from their chosen strategy. In this setting, each player's strategy is nearly optimal, with at most a margin of $\epsilon$ in improvement.

We now state the main assumptions from \cite{coh-zel2022} that are needed in order to obtain the asymptotic optimality of the profiles we would work with.

\begin{enumerate}[label=\textbf{(A\arabic*)}]
    \item For all $x \neq y \in [d]$, the set of allowed transition rates from state $x$ to $y$ is $[L,U]$ for some $0 < L < U$.
    
	\item The value of the individual running cost function $f(x,a)$ is independent of the value of $a_x$. Also, the interaction $F(x, \cdot)$ is in \textbf{(D-Lip)} for all $x \in [d]$, and Lasry--Lions monotone, i.e. for all $\mu, \nu \in S_d$,
		\begin{equation*}
			\sum_{x \in [d]} (F(x, \mu) - F(x, \nu))(\mu_x - \nu_x) \geq 0.
		\end{equation*}

	\item The Hamiltonian $H \from [d] \times \RR^d \to \RR$, defined by 
    \[
    H(x,p)=\min_a\{f(x,a)+a\cdot p\},
    \]
    is twice continuously differentiable with respect to $p$, the gradient $D_p H$ and the Hessian $D^2_{pp} H$ are Lipschitz in $p$ and there exists $C_{2, H} > 0$ such that $D^2_{pp}H(x, p) \leq -C_{2, H}$.
\end{enumerate}
We define $\gamma^*(x,p) \defeq D_p H(x,p)$, which (by properly defining $\gamma^*_x(x,p)=-\sum_{y\ne x}\gamma^*_y(x,p)$) is the vector of transition rates out of the state $x$, shown below in \eqref{eq:ergMFG}.

A mean field equilibrium is characterized as a fixed point of the best response mapping for a representative player minimizing an ergodic cost. We omit the formal definition of the mean field equilibrium, as it is not central to the following discussion. Instead, we focus on the relevant equations of the mean field game from which the asymptotic Nash equilibria are derived.

\subsubsection{Two Asymptotically Nash Equilibria Regimes Driven by MFG Equations}
MFGs are typically described by a forward-backward coupled system of differential equations. In this system, the backward equation is an HJB equation that describes the value function, while the forward equation governs the evolution of the state distribution of the representative player under the mean field equilibrium. Another key equation in this framework is the master equation of the MFG, which encapsulates the forward-backward system into a single differential equation. This formulation has been employed, for instance, in \cite{CardaliaguetDelarueLasryLions}, \cite{bay-coh2017}, and \cite{cec-pel2019} to analyze the convergence of Nash equilibria in $N$-player games to the mean field equilibrium.

In contrast, the ergodic case presents a different scenario. Due to the stationary nature of the problem, we have a stationary (algebraic) MFG system, which consists of a stationary HJB equation and a distribution equation. This system characterizes the stationary state alone. This master equation of the MFG, however, is more general. It incorporates a feedback structure and can describe states beyond the stationary regime, offering a broader perspective. We will now delve into more detailed aspects of this distinction, taking the equations from \cite{coh-zel2022, coh-zel24}.

We start with the \textit{stationary ergodic MFG system} is
\begin{align}\label{eq:statErgMFG}
	\begin{cases}
		- \bar\varrho + H(x, \Delta_x \overline{u}) + F(x, \overline{\mu}) = 0, \\
		\sum_{y \in [d]} \overline{\mu}_y \gamma^*_x(y, \Delta_y \overline{u}) = 0.\\
	\end{cases}
\end{align}
The solution of \eqref{eq:statErgMFG} is a triple, $(\bar\varrho, \overline{u}, \overline{\mu})$, which under the assumptions above is unique. It has the interpretation that the {\it value of the MFG} (the cost under the ergodic mean field equilibrium) is given by $\bar\varrho \in \RR$, the potential vector is $\overline{u} \in \RR^d_0$, and the stationary distribution of a representative player in the mean field game is $\overline{\mu} \in S_d$.
The potential vector $\overline{u}$ represents the ``cost-to-go" from a given state, incorporating both the immediate cost and the long-term behavior of the system. It can be thought of as analogous to the value function in the finite-horizon case, but is more precisely the limit of differences of the value functions for the discounted mean field game as the discount rate goes to zero.
It is unique up to an additive contact vector: $(c,\ldots,c)$ for any $c\in\mathbb{R}$. The transition rate $\gamma^*_z(y, \Delta_y \bar u)$ from state $y$ to $z$, which is independent of $t$ and $\mu$, governs the mean field equilibrium of the representative player in the MFG.

We also consider the \textit{ergodic master equation of the MFG}: 
\begin{equation}\label{eq:ergMaster}
	\bar{\varrho} = H(x, \Delta_x U^0(\cdot, \mu)) + F(x, \mu) + \sum_{y, z \in [d]} \mu_y D^\frm_{yz} U^0(x, \mu) \gamma^*_z(y, \Delta_y U^0(\cdot, \mu)),
\end{equation}
which has solutions of the form $(\bar{\varrho}, U^0)$, where $\bar{\varrho} \in \RR$ is again the value of the MFG, $U^0 \from [d] \times S_d \to \RR$ is the master equation potential function (like the ``cost-to-go" $\bar u$, but in a feedback form), and $\gamma^*_z(y, \Delta_y U^0(\cdot, \mu))$ is the transition rate from state $y$ to $z$, when the current distribution of the state of the representative player is $\mu$, in a mean field equilibrium.

Altogether, we presented two transition rates that govern the mean-field equilibrium of the representative player in the MFG. However, we do not discuss the mean-field equilibria in this paper, as it is not directly relevant to the results we will prove. Instead, we will use these rates in an $N$-player game setup, as described below. Before doing so, we first establish the connection between the two.

By vanishing discount type arguments, Cohen and Zell have shown in \cite{coh-zel2022} that solutions to \eqref{eq:statErgMFG} correspond to mean field equilibria and that by plugging in the stationary distribution in the master equation of the MFG, namely $\mu=\bar\mu$, one restores \eqref{eq:ergMaster}. Specifically, $\gamma^*_z(y, \Delta_y U^0(\cdot, \bar\mu))=\gamma^*_z(y, \Delta_y \bar u)$ and up to a constant additive factor in the potential, solutions to \eqref{eq:ergMaster} correspond to solutions of the stationary ergodic MFG system.

In \cite[Corollary 3.1]{coh-zel2022} it is shown that the {\it independent control profile} or {\it MFG system derived profile}:
\begin{align*}
         \bar\Gamma:=(\bar\gamma)_{i\in [N]} \quad \text{ where } \quad \bar\gamma := [\gamma^*_y(x, \Delta_x \bar u)]_{x,y\in [d]},
\end{align*}
achieves an $\mathcal{O}(N^{-1/2})$-Nash equilibrium and that the costs to the players under this profile are asymptotically $\bar\varrho$:
\begin{align}\notag
    |J^{i,N} (\boldsymbol{X}^N_0, \bar \Gamma) - \bar \varrho| \leq \frac{C}{\sqrt{N}}.
    \end{align} 
Also, if $(X^{i,N}_0)_{i\in [n]}$ are exchangeable, though not necessarily independent, with $X^{i,N}_0 \sim \bar\mu$ for all $i\in [N]$, then:
    \begin{align}\label{eq:barGamma}
        \sup_{t\in[0,\infty)^{d-1}]} \mathbb{E} |\mu^{\bar \Gamma}_t - \bar \mu| \leq \frac{C}{\sqrt{N}},
    \end{align}
where $\mu^{\bar \Gamma}_t$ is the empirical distribution of the players at time $t$ under the profile $\bar\Gamma$. Note that we are comparing the empirical distribution at time $t$ to the stationary distribution in the MFG, given in \eqref{eq:statErgMFG}.

Additionally, \cite{coh-zel2022} considers also the {\it mass-dependent controls} or {\it master equation derived profile}: 
\begin{align*}
         \Gamma^0:=(\Gamma^0)_{i\in [N]} \quad \text{ where } \quad \Gamma^0 := [\gamma^*_y(x, \Delta_x U^0(\cdot,\mu_t^{\Gamma^0}))]_{x,y\in [d]},
\end{align*}
where $(\mu_t^{\Gamma^0})_{t\ge0}$ is the empirical distribution process under this profile.  In \cite[Theorem 2]{coh-zel2022} it is shown that this profile 
achieves an $\mathcal{O}(N^{-1/2})$-Nash equilibrium and that the costs to the players under this profile are asymptotically $\bar\varrho$:
\begin{align}\notag
    |J^{i,N} (\boldsymbol{X}^N_0,  \Gamma^0) - \bar \varrho| \leq \frac{C}{\sqrt{N}}.
    \end{align} 
Also, if $(X^{i,N}_0)_{i\in [n]}$ are exchangeable, though not necessarily independent, with $X^{i,N}_0 \sim \mu$ for all $i\in [N]$, then:
    \begin{align*}
        \lim_{t\to\infty} \mathbb{E} |\mu^{ \Gamma^0}_t - \bar \mu| \leq \frac{C}{\sqrt{N}}.
    \end{align*}
Note that this result is slightly different from \eqref{eq:barGamma}. Here we compare the {\it limit} of the empirical distribution to the stationary one, and we allow for initial distributions other than the stationary distribution. Next, we demonstrate how our main result, \Cref{thm:mainRes}, strengthens these findings in both regimes.

\subsubsection{Uniform in Time Convergence of the Empirical Distributions}

We present the following theorem for the mass-dependent controls. The results for the case with independent controls are simpler to prove, as the players' dynamics are independent.

\begin{thm}
	Assuming \textbf{(A1)}, \textbf{(A2)}, and \textbf{(A3)}, the $N$-player system with players using the master equation of the MFG derived strategies satisfies the assumptions of \cref{thm:mainRes}, and so, taking $m(t; \mu)$ to be the solution to \eqref{eq:kolmogorov} with $\alpha_{xy}(\mu) = \gamma^*_y(x, \Delta_x U^0(\cdot, \mu))$, we have that the empirical distributions of the $N$-player games converge weakly to $m(t; \mu)$ at rate $1/N$.
\end{thm}

\begin{proof}

    As the regularity assumptions on $\alpha$ are included in \textbf{(A2)}, all that remains to apply \cref{thm:mainRes} is to show \textbf{(Erg)}, which we do in three steps.
    In the first, we provide the time-dependent ergodic MFG system that includes
    \eqref{eq:kolmogorov} with $\alpha$ as defined in the statement of the theorem.
    We then establish the necessary regularity and exponential decay of the measure flow for a relevant discounted MFG system.
    Finally, we use a vanishing discount argument to transfer these properties to the solution of the ergodic MFG system, establishing \textbf{(Erg)}.
    
    \vspace{5pt}
    \noindent\textbf{Step 1:}
    We introduce the \textit{time-dependent ergodic MFG system}, given in \cite{coh-zel24}, which is given by the following system of equations:
    \begin{align}\label{eq:ergMFG}
    	\begin{cases}
    		-\frac{d}{dt}\check{u}(t, x) + \check{\varrho} = H(x, \Delta_x \check{u}(t, \cdot)) + F(x, \check{\mu}), \\
    		\frac{d}{dt}\check{\mu}(t, x) = \sum_{y \in [d]} \check{\mu}_y \gamma^*_x(y, \Delta_y \check{u}(t, \cdot)),\\
    		\check{\mu}(0) = \mu\in S^d.
    	\end{cases}
    \end{align}
    A solution of \eqref{eq:ergMFG} is a triple $(\check{\rho}, \check{u}, \check{\mu})$ with $\check\varrho\in\mathbb{R}$, $\check{u}:[0,\infty)\times[d]\to\mathbb{R}$, and $\check\mu:[0,\infty)\times[d]\to S_d$.\footnote{Note that \eqref{eq:statErgMFG} is also \eqref{eq:ergMFG} with $\mu = \overline{\mu}$.}
	
    Taking $u(t, x) \defeq U^0(x, m(t; \mu))$, we will first show that $(\bar \varrho, u, m(t; \mu))$ solves \eqref{eq:ergMFG}.
	By \eqref{eq:chain1} and \eqref{eq:derRel1}, we have that 
	\begin{align*}
		\frac{d}{dt} u(t,x) &= \sum_{z \in [d]} D^\frm_{1z}U^0(x, m(t; \mu)) \frac{d}{dt} m_z(t; \mu) \\
							&= \sum_{z \in [d]} D^\frm_{1z}U^0(x, m(t; \mu)) \sum_{y \in [d]} m_y(t; \mu) \gamma^*_z(y, \Delta_y U^0(\cdot, m(t; \mu))) \\
							&= \sum_{y, z \in [d]} m_y(t; \mu) D^\frm_{y 1}U^0(x, m(t; \mu)) \gamma^*_z(y, \Delta_y U^0(\cdot, m(t; \mu))) \\
							&\qquad + \sum_{y, z \in [d]} m_y(t; \mu) D^\frm_{1z}U^0(x, m(t; \mu)) \gamma^*_z(y, \Delta_y U^0(\cdot, m(t; \mu))),
	\end{align*}
	where the last equality follows from the fact that $\gamma^*_z$ takes values in $\RR^d_0$.
	Using that for any differentiable function $F \from S_d \to \RR$ we have $D^\frm_{yz}F = D^\frm_{y 1}F + D^\frm_{1z}F$, we can rewrite this as
	\begin{align*}
		\frac{d}{dt} u(t,x) &= \sum_{y, z \in [d]} m_y(t; \mu) D^\frm_{yz}U^0(x, m(t; \mu)) \gamma^*_z(y, \Delta_y U^0(\cdot, m(t; \mu))) \\
							&= \bar\varrho - H(x, \Delta_x U^0(\cdot, m(t; \mu))) - F(x, m(t; \mu)),
	\end{align*}
	where we have used \eqref{eq:ergMaster} in the last step.
	This is exactly the first equation in \eqref{eq:ergMFG}, and so we have that $(\bar\varrho, u, m(t; \mu))$ solves \eqref{eq:ergMFG}.

    \vspace{5pt}
    \noindent\textbf{Step 2:} 
    We now prove regularity results for solutions of the \textit{time-dependent discounted MFG system}:
    \begin{align}\label{eq:discMFG}
    	\begin{cases}
    		-\frac{d}{dt}u^r(t, x) = -r u^r(t, x) + H(x, \Delta_x u^r(t, \cdot)) + F(x, \mu^r)), \\
    		\frac{d}{dt}\mu^r(t, x) = \sum_{y \in [d]} \mu^r_y \gamma^*_x(y, \Delta_y u^r(t, \cdot)), \\
    		\mu^r(0) = \mu,
    	\end{cases}
    \end{align}
    which additionally requires the \textit{discounted master equation of the MFG}:
    \begin{equation}\label{eq:discMaster}
    	rU^r(x, \mu) = H(x, \Delta_x U^r(\cdot, \mu)) + F(x, \mu) + \sum_{y, z \in [d]} \mu_y D^\frm_{yz} U^r(x, \mu) \gamma^*_z(y, \Delta_y U^r(\cdot, \mu)),
    \end{equation}
    with $r>0$, both of which were studied in \cite{coh-zel2022}. 
    Taking, as in the ergodic case, $\Tilde{u}^r(t, x) \defeq U^r(x, m^r(t; \mu))$, where $m^r(t; \mu)$ is a solution to \eqref{eq:kolmogorov} with $\alpha^r_{xy}(\mu) = \gamma^*_y(x, \Delta_x U^r(\cdot, \mu))$, it follows by the same argument as in \textbf{Step 1} that $(\Tilde{u}^r, m^r)$ is a solution to \eqref{eq:discMFG} with the initial condition $m^r(0; \mu) = \mu$.
    By \cite[Prop. 2.2 (iii)]{coh-zel2022}, there exists $r_0 > 0$ such that for all discount factors $r \in (0, r_0)$ the solutions $(u^r, m^r(t; \mu))$ and $(\hat{u}^r, m^r(t; \hat{\mu}))$ to the discounted MFG system \eqref{eq:discMFG} satisfy
	\begin{equation}\label{eq:discExpStab}
		\lvert m^r(t; \mu) - m^r(t; \hat{\mu}) \rvert \leq C e^{-\lambda t} \lvert \mu - \hat{\mu} \rvert.
	\end{equation}
    We want to establish a similar result to the above for the derivatives of $m^r$ with respect to $\mu$.
    Precisely, we will show that
    \begin{equation}\label{eq:discDerExpStab}
        \left\lvert \frac{\delta m^r}{\delta \frm}(t, \hat{\mu}, z) - \frac{\delta m^r}{\delta \frm}(t, \mu, z) \right\rvert \leq C e^{-\theta t} \lvert \hat{\mu} - \mu \rvert.
    \end{equation}
    Our argument requires a linearized system around $(u^r, m^r(t; \mu))$ given by
    \begin{align}\label{eq:discLinSystem}
    \begin{cases}
        &-\frac{d}{dt}v_x(t) = - rv_x(t) + \gamma^*(x, \Delta_x u^r(t)) \cdot \Delta_x v(t) + D^\frm_1 F(x, m^r(t; \mu)) \cdot q(t), \\
        &\frac{d}{dt} q_x(t) = \sum_{y \in [d]} q_y(t) \gamma^*_x(y, \Delta_y u^r(t)) + \sum_{y \in [d]} m^r_y(t; \mu) \nabla_p \gamma^*_x(y, \Delta_y u^r(t)) \cdot \Delta_y v(t), \\
        &q(0) = q_0 \in \RR^d_0, \qquad v \text{ is bounded}.
    \end{cases}
    \end{align}

    By \cite[Prop. 5.1]{coh-zel2022} the system above has a solution $(v, q)$ for $r$ sufficiently small. Also, by \cite[Cor. 7.1]{coh-zel2022} there exist constants $C, \theta > 0$ independent of $q_0$, $t$, and $r$, such that
    \begin{equation}\label{eq:qDecay}
        \lvert q(t) \rvert + \lvert \Delta v(t) \rvert + \lvert v(t) \rvert \leq C e^{-\theta t} \lvert q_0 \rvert.
    \end{equation}

    If we can show that
    \begin{equation}\label{eq:linComp}
        q(t) = \frac{\delta m^r}{\delta \frm}(t, \mu, z)
    \end{equation}
    for the appropriate choice of $q_0$, then \eqref{eq:qDecay} will give us the desired exponential stability \eqref{eq:discDerExpStab}, since for $(v, q), (\hat{v}, \hat{q})$ solutions of \eqref{eq:discLinSystem} with initial conditions $q_0$, and $\hat{q}_0$ respectively, we have that $(v - \hat{v}, q - \hat{q})$ solves \eqref{eq:discLinSystem} with the initial condition $q_0 - \hat{q}_0$.

    By \cite[Lem. 7.1]{coh-zel2022}, we know that for $(u^r, m^r(t; \mu))$ and $(\hat{u}^r, m^r(t; \hat{\mu}))$ solutions of \eqref{eq:discMFG} and $(v, q)$ the solution of \eqref{eq:discLinSystem} with $q_0 = \mu - \hat{\mu}$, we have
    \begin{equation*}
        \sup_{t \in [0, \infty)} \lvert m^r(t, \hat{\mu}) - m^r(t; \mu) - q(t) \rvert \leq C \lvert \mu - \hat{\mu} \rvert^2.
    \end{equation*}
    We use this to show agreement in \eqref{eq:linComp}. 
    Thus, for each $z \in [d]$, we let $(v^{(z)},q^{(z)})$ be the solution to \eqref{eq:discLinSystem} with $q^{(z)}(0) = \delta_z - \mu$ and $(v^{(z)}_{\varepsilon}, q^{(z)}_{\varepsilon})$ be the solution to \eqref{eq:discLinSystem} with $q^{(z)}_{\varepsilon}(0) = \varepsilon(\delta_z - \mu)$, and obtain
    \begin{equation*}
        \frac{\delta m^r}{\delta \frm}(t, \mu, z) = \lim_{\varepsilon \to 0} \frac{m^r(t; (1-\varepsilon)\mu + \varepsilon \delta_z) - m^r(t; \mu)}{\varepsilon} = \lim_{\varepsilon \to 0} \frac{q^{(z)}_{\varepsilon}(t)}{\varepsilon}.
    \end{equation*}
    It is straightforward to see that in fact $q^{(z)}_{\varepsilon}(t) = \varepsilon q^{(z)}(t)$, so we obtain
    \begin{equation*}
        \frac{\delta m^r}{\delta \frm}(t, \mu, z) = q_z(t),
    \end{equation*}
    showing the agreement between the two forms of linearized system, and by \eqref{eq:qDecay} that we have \eqref{eq:discDerExpStab}.

    \vspace{5pt}
    \noindent\textbf{Step 3:} 
    We now show by a vanishing discount argument that these results transfer to $m(t; \mu)$.

    By the proof of Theorem 2.1 (i) in \cite{coh-zel2022}, we know that there exists a sequence of discount factors $(r_n)_n$ converging to $0$ such that
    \begin{equation*}
        U^{r_n}(x, \cdot) - U^{r_n}(1, \cdot) \to U^0(x, \cdot)
    \end{equation*}
    and
    \begin{equation*}
        D^\frm U^{r_n}(x, \cdot) \to D^\frm U^0(x \cdot)
    \end{equation*}
    both uniformly in the measure argument.
    We claim that this shows convergence of $m^{r_n}(t; \mu)$ to $m(t; \mu)$ pointwise in $t$ and uniformly in $\mu$.
    Defining $\alpha^{r_n}_{xy}(\mu) \defeq \gamma^*_y(x, \Delta_x U^{r_n}(\cdot, \mu))$ and fixing $t \geq 0$, we have that
    \begin{align*}
        \lvert m^{r_n}_x(t; \mu) - m_x(t; \mu) \rvert &\leq \int_0^t \bigg\lvert \sum_{y \in [d]} \Big[ m^{r_n}_y(s; \mu) \alpha^{r_n}_{yx}(m^{r_n}(s; \mu)) - m_y(t; \mu) \alpha_{yx}(m(t; \mu)) \Big] \bigg\rvert ds \\
        & \leq \int_0^t \bigg\{\bigg\lvert \sum_{y \in [d]} m^{r_n}_y(s; \mu) \alpha^{r_n}_{yx}(m^{r_n}(s; \mu)) - m^{r_n}_y(s; \mu) \alpha_{yx}(m^{r_n}(s; \mu)) \bigg\rvert \\
        &\qquad + \bigg\lvert \sum_{y \in [d]} m^{r_n}_y(s; \mu) \alpha_{yx}(m^{r_n}(s; \mu)) - m_y(t; \mu) \alpha_{yx}(m(t; \mu)) \bigg\rvert\bigg\} ds.
    \end{align*}
    The Lipschitz continuity of $\alpha$ gives that the second term in the integral is bounded by $C \lvert m^{r_n}(s; \mu) - m(s; \mu) \rvert$ where $C$ is a constant independent of $\mu$, $s$, $t$ and $r_n$.
    Then the uniform convergence of $\Delta_y U^{r_n}$ to $\Delta_y U^0$, together with the Lipschitz continuity of both $\alpha^{r_n}$ and $\alpha$, gives that there exists $N$ such that for all $n \geq N$, the first term is bounded by $\varepsilon$.
    Thus we obtain that, for $n \geq N$,
    \begin{equation*}
        \lvert m^{r_n}_x(t; \mu) - m_x(t; \mu) \rvert \leq t\varepsilon + \int_0^t C \lvert m^{r_n}_x(s; \mu) - m_x(s; \mu) \rvert ds.
    \end{equation*}
    Applying Gronwall's inequality, we get that
    \begin{equation*}
        \lvert m^{r_n}_x(t; \mu) - m_x(t; \mu) \rvert \leq t\varepsilon + \int_0^t s \varepsilon C e^{C(t-s)} ds,
    \end{equation*}
    which shows that for any fixed $t$, $m^{r_n}(t; \mu)$ converges uniformly to $m(t; \mu)$.
    Thus we have by \eqref{eq:discExpStab} that $m(t; \mu)$ satisfies \eqref{eq:convergence}.

    Applying the same procedure to the linearized equation satisfied by $\frac{\delta m^{r_n}}{\delta \frm}$ shows that it converges pointwise in $t$ and uniformly in $\mu$ to $\frac{\delta m}{\delta \frm}$ and so by \eqref{eq:discDerExpStab}, we get that $\frac{\delta m}{\delta \frm}$ satisfies \eqref{eq:derLip}.
    Thus, we have shown that the mass-dependent controls satisfy \textbf{(Erg)} and so we may apply \cref{thm:mainRes}.

\end{proof}

\appendix
\section{Chain Rule for Functional Derivatives}\label{sec:chainRule}

\begin{prop}\label{prop:chainRule}
	For functions $F \from S_d \to \RR$, $\varphi \from \RR \to S_d$, and $\psi \from S_d \to S_d$, each everywhere functional differentiable, we have
	\begin{equation}\label{eq:chain1}
		\frac{d}{dt} (F\circ\varphi)(t) = \frac{\delta F}{\delta \frm}(\varphi(t), \cdot) \cdot \frac{d}{dt} \varphi(t),\qquad t\in\RR,
	\end{equation}
	and
	\begin{equation}\label{eq:chain2}
		\frac{\delta (F\circ\psi)}{\delta \frm}(\mu, z) = \frac{\delta F}{\delta \frm}(\psi(\mu), \cdot) \cdot \frac{\delta \psi}{\delta \frm}(m, z),\qquad (\mu,t)\in S_d\times\RR.
	\end{equation}
\end{prop}

\begin{proof}
	We first prove \eqref{eq:chain1}.
	We see that, defining $\tilde{F}$ and $\tilde{\varphi}$ as usual, we have for any $t\in\RR$,
	\begin{equation*}
		F(\varphi(t)) = \tilde{F}(\tilde{\varphi}(t)),
	\end{equation*}
	and thus
	\begin{equation*}
		\frac{d}{dt} F(\varphi(t)) = \frac{d}{dt} \tilde{F}(\tilde{\varphi}(t)).
	\end{equation*}
	By \cite[Proposition 5.66]{car-del2018a}, we have that \begin{align}\label{eq:CD5.66}
 \frac{\partial\tilde{F}}{\partial m_z} (\tilde{\mu}) = D^m_{dz}F(\mu),\qquad \tilde\mu\in\tilde S_d,
 \end{align}
 which, combined with the chain rule, gives
	\begin{align*}
		\frac{d}{dt} (F\circ\varphi)(t) &= \sum_{z \in [d-1]} \frac{\partial \tilde{F}}{\partial m_z} (\tilde{\varphi}(t)) \frac{d}{dt} \tilde{\varphi}_z(t) \\
								   &= \sum_{z \in [d-1]} \left( \frac{\delta F}{\delta \frm}(\varphi(t), z) - \frac{\delta F}{\delta \frm}(\varphi(t), d) \right) \frac{d}{dt} \varphi_z(t) \\
								   &= \sum_{z \in [d]} \frac{\delta F}{\delta \frm}(\varphi(t), z) \frac{d}{dt} \varphi_z(t),
	\end{align*}
	where the last equality follows from the fact that 
    \begin{equation*}
        \sum_{z \in [d]} \frac{d}{dt} \varphi_z(t) = 0,
    \end{equation*}
    which in turn follows since $\sum_{z\in[d]}\varphi_z(t)= 1$. This concludes the proof of \eqref{eq:chain1}.

	We now prove \eqref{eq:chain2}.
	Again, we define $\tilde{F}$ and $\tilde{\psi}$ as usual, and see that for any $\tilde\mu\in\tilde S_d$,
	\begin{equation*}
		\frac{\partial }{\partial m_z} (\tilde{F}\circ\tilde{\psi})(\tilde{\mu}) = \sum_{y \in [d-1]}  \frac{\partial \tilde{F}}{\partial m_y}  (\tilde{\psi}(\tilde{\mu})) \frac{\partial \tilde{\psi}_y}{\partial m_z} (\tilde{\mu}).
	\end{equation*}
	Applying \eqref{eq:CD5.66}, this is equal to
	\begin{equation*}
		\sum_{y \in [d-1]} D^m_{dy}F(\psi(\mu)) D^m_{dz}\psi_y(\mu),
	\end{equation*}
	which we can expand to get
	\begin{equation*}
		\sum_{y \in [d]} \frac{\delta F}{\delta \frm}(\psi(\mu), y) \left(\left( \frac{\delta \psi}{\delta \frm}(\mu, z)\right)_y - \left(\frac{\delta \psi}{\delta \frm}(\mu, d)\right)_y \right),
	\end{equation*}
	where similarly to before, we have used the fact that $\sum_{y \in [d]} \left(\frac{\delta \psi}{\delta \frm}(\mu, z)\right)_y = 0$.

	Replacing $\tilde{F}$ with $\tilde{F}\circ\tilde{\psi}$ in \eqref{eq:CD5.66}, we get that for any $z \in [d-1]$,
	\begin{equation}\label{eq:chain2proof}
		\frac{\delta (F\circ\psi)}{\delta \frm}(\mu, z) - \frac{\delta (F\circ\psi)}{\delta \frm}(\mu, d) = \frac{\delta F}{\delta \frm}(\psi(\mu), \cdot) \cdot \left( \frac{\delta \psi}{\delta \frm}(\mu, z) - \frac{\delta \psi}{\delta \frm}(\mu, d) \right).
	\end{equation}
	Next, we multiply both sides by $\mu_z$ and sum over $z$ to get
	\begin{equation*}
		-\frac{\delta (F\circ\psi)}{\delta \frm}(\mu, d)\mu_d + (\mu_d - 1) \frac{\delta (F\circ\psi)}{\delta \frm}(\mu, d) = \frac{\delta F}{\delta \frm}(\psi(\mu), \cdot) \cdot \left( - \frac{\delta \psi}{\delta \frm}(\mu, d)\mu_d + (\mu_d - 1)\frac{\delta \psi}{\delta \frm}(\mu, d) \right).
	\end{equation*}
	Simplifying, this is equal to
	\begin{equation*}
		\frac{\delta (F\circ\psi)}{\delta \frm}(\mu, d) = \frac{\delta F}{\delta \frm}(\psi(\mu), \cdot) \cdot \frac{\delta \psi}{\delta \frm}(\mu, d).
	\end{equation*}
	We then substitute this back into \eqref{eq:chain2proof} to obtain \eqref{eq:chain2}.
\end{proof}

{\bf Acknowledgment.}  We thank the anonymous Associate Editor and the anonymous reviewers for their valuable reports and comments which helped to improve the quality of this paper. A. Cohen gratefully acknowledges support from NSF under grants DMS-2505998.

\footnotesize

\bibliographystyle{abbrv} 
\bibliography{references} 
\end{document}